\documentclass[a4paper, 11pt, reqno]{amsart}
\usepackage{amsmath}
\usepackage{amssymb}
\usepackage{amsthm}
\usepackage{bm}
\usepackage[all]{xy}
\usepackage{mathrsfs}
\usepackage{mathtools}
\usepackage{stmaryrd}
\usepackage[noadjust]{cite}
\numberwithin{equation}{section}
\newtheorem{thm}{Theorem}[section]
\newtheorem{defn}[thm]{Definition}
\newtheorem{lem}[thm]{Lemma}
\newtheorem{prop}[thm]{Proposition}
\newtheorem{cor}[thm]{Corollary}

\theoremstyle{remark}
\newtheorem{rmk}[thm]{Remark}

\theoremstyle{definition}
\newtheorem{exa}[thm]{Example}

\begin{document}

\title{Wide subcategories of a domestic weighted projective line}
\author{Yiyu Cheng}

\begin{abstract}
   For a weighted projective line $\mathbb X$, a wide subcategory of the category $\mathrm{coh} \text{-} \mathbb X$ of coherent sheaves over $\mathbb X$ is called $\vec c$-invariant if it is closed under the grading shift of the canonical element $\vec c$.
   We proved that a $\vec c$-invariant wide subcategory of $\mathrm{coh} \text{-} \mathbb X$ containing a vector bundle is the perpendicular category of a torsion exceptional sequence. If $\mathbb X$ is of dometic type, then the poset of wide subcategories of $\mathrm{coh} \text{-} \mathbb X$ is the union of the poset of wide subcategories generated by an exceptional sequence and the poset of $\vec c$-invariant wide subcategories.
\end{abstract}

\maketitle

\section{Introduction}

Wide subcategories of a hereditary abelian category have been extensively studied in representation theory of algebras. According to \cite[Theorem~4.7]{ingalls2009noncrossing}, for a representation-finite hereditary algebra $A$, the poset of wide subcategories of $A$-mod, which is ordered by inclusion, is isomorphic to the lattice of non-crossing partitions for the Weyl group of $A$. We recall that non-crossing partitions were introduced in \cite{kreweras1972partitions} and generalised in the context of Coxeter groups in \cite{brady2001partial,brady2002k}. The theorem mentioned above was generalized to any hereditary algebra in \cite[Theorem~1.2]{hubery2016categorification}, but only wide subcategories generated by an exceptional sequence were considered.

We study the poset of all wide subcategories for the category of coherent sheaves over a domestic weighted projective line. It is well known that a domestic weighted projective line is derived equivalent to  a tame hereditary algebra; see \cite[Subsection~5.4.1]{geigle1987class} and \cite[Theorem~3.5]{lenzing2011weighted}. So in fact, the poset of all wide subcategories for the category of modules over a tame hereditary algebra is studied.

In \cite[Subsection~4.1]{krause2017derived}, the authors described the poset(lattice) of wide subcategories of $\mathrm{coh} \text{-} \mathbb P^1$, the category of coherent sheaves over the projective line. It is isomorphic to the coproduct of lattices
\[
   \{\text{specialization closed subsets of } \mathbb P^1 \} \coprod \mathbb Z,
\]
where $\mathbb Z$ denotes the lattice given by the following Hasse diagram:

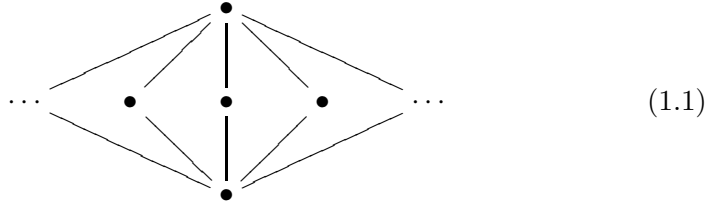
\begin{equation} \label{lattice_z}
   \begin{split}
      \xymatrix{
 & & \bullet \ar@{-}[dll] \ar@{-}[dl] \ar@{-}[d] \ar@{-}[dr] \ar@{-}[drr] & & \\ 
 \cdots & \bullet & \bullet & \bullet & \cdots \\  
 & & \bullet \ar@{-}[ull] \ar@{-}[ul] \ar@{-}[u] \ar@{-}[ur] \ar@{-}[urr] & &    }
   \end{split}
\end{equation}
Note that $\mathbb Z$ corresponds to the lattice of wide subcategories which can be generated by an exceptional sequence.
Since $\mathbb P^1$ is a domestic weighted projective line, our work generalises this example.

For a weighted projective line $\mathbb X$ and the category $\mathrm{coh}\text{-}\mathbb X$ of coherent sheaves over $\mathbb X$, we donote by $\mathrm{Wid}(\mathrm{coh} \text{-} \mathbb X)$ the poset of wide subcategories of $\mathrm{coh} \text{-} \mathbb X$, and $\mathrm{Exc}(\mathrm{coh} \text{-} \mathbb X)$ the poset of those can be generated by an exceptional sequence, and $\mathrm{Wid}(\mathrm{coh} \text{-} \mathbb X)^{\vec c}$ the poset of $\vec c$-invariant wide subcategories. Here $\vec c$ is the canonical element of the grading group of $\mathbb X$, and a subcategory is called $\vec c$-invariant if it is closed under taking grading shift of degree $\vec c$. We have the following main theorem. 

\begin{thm}
   (See Theorem~\ref{domestic_order}) Let $\mathbb X$ be a weighted projective line of domestic type. Then $\mathrm{Wid}(\mathrm{coh} \text{-} \mathbb X)$ can be presented as a pushout of posets
   \[
   \xymatrix{
   \mathrm{Exc}(\mathrm{coh} \text{-} \mathbb X) \cap \mathrm{Wid}(\mathrm{coh} \text{-} \mathbb X)^{\vec c} \ar@{^(->}[r] \ar@{^(->}[d] & \mathrm{Wid}(\mathrm{coh} \text{-} \mathbb X)^{\vec c} \ar@{^(->}[d]\\
   \mathrm{Exc}(\mathrm{coh} \text{-} \mathbb X) \ar@{^(->}[r] & \mathrm{Wid}(\mathrm{coh} \text{-} \mathbb X).
   }
   \]
\end{thm}

For a weighted projective line $\mathbb X$ of any type, we have the following classification theorem for $\vec c$-invariant wide subcategories. And by the theorem, the poset structure of $\mathrm{Wid}(\mathrm{coh} \text{-} \mathbb X)^{\vec c}$ can be described.

\begin{thm}
   (See Theorem \ref{c_inv}) Let $\mathcal W$ be a $\vec c$-invariant wide subcategory of $\mathrm{coh} \text{-} \mathbb X$. Then $\mathcal W$ is a perpendicular category of a torsion exceptional sequence. 
\end{thm}

% \begin{prop}
%    (See Proposition~\ref{thick_domestic_bundle}) Let $\mathbb X$ be a weighted projective line of domestic type and $\mathcal W$ a wide subcategory of $\mathrm{coh} \text{-} \mathbb X$. Then $\mathcal W$ is either generated by an exceptional sequence, or formed by only torsion sheaves.
% \end{prop}

In Section 2 and 3, we arrange some known properties of hereditary categories, which are mainly about tilting objects, exceptional sequences and perpendicular categories. 

In Section 4, we recall some basic properties of tubes, which forms the subcategory of torsion sheaves of a weighted projective line. We refer to \cite{dichev2009thick} for further information. And we refer to \cite{krause2021category} for the lattice structure of wide subcategories of a tube: it is isomorphic to the lattice of non-crossing partitions of type $B$. 

In Section 5, we recall weighted projective lines and category of coherent sheaves. 

In Section 6, we introduce $\vec c$-invariant wide subcategories of a weighted projective line and give the classification theorem of them. 

In Section 7, we classify wide subcategories of a domestic weighted projective line, and give our main theorem.

In this paper, all categories are supposed to be essentially small, i.e. the isomorphism classes of objects form a set.

\section{Hereditary categories and wide subcategories}

Let $\mathcal A$ be an abelian category, and $\mathcal W$ be a full subcategory of $\mathcal A$. We call $\mathcal W$ a \emph{wide} subcategory if it is closed under taking kernels, cokernels, and extensions.

An abelian category $\mathcal H$ is called \emph{hereditary}, if $\mathrm{Ext}_{\mathcal H}^2(X,Y)=0$ for any $X,Y \in \mathcal H$. For a hereditary abelian category $\mathcal H$, we denote by $\mathcal D^b(\mathcal H)$ the bounded derived category of $\mathcal H$. A triangulated subcategory $\mathcal D$ of $\mathcal D^b(\mathcal H)$ is called \emph{thick} if it is closed under taking direct summands.

\begin{prop} [{\cite[Theorem~5.1]{bruning2007thick}}] \label{bruning}
Let $\mathcal H$ be a hereditary abelian category. Then the following statements hold.\\
(1) For any $X \in \mathcal D^b(\mathcal H)$, we have a non-canonical isomorphism
\[
X \cong \bigoplus_{n \in \mathbb Z}H^nX[-n].
\] \\
(2) We have an one-to-one correspondence between wide subcategories of $\mathcal H$ and thick subcategories of $\mathcal D^b(\mathcal H)$, by assignments
\[
\mathcal W \mapsto \{X \in \mathcal D^b(\mathcal H) \; | \; H^n(X) \in \mathcal W, \ \forall n \in \mathbb Z \} \ and \  \mathcal T \mapsto \{ H^0(X) | X \in \mathcal T \}.
\]
\end{prop}

Let $\mathcal H$ be a hereditary abelian category, and $\mathcal X \subseteq \mathcal H$ be a class of objects. We denote by $\langle \mathcal X \rangle$ the smallest wide subcategory of $\mathcal H$ containing $\mathcal X$, which is called the wide subcategory \emph{generated} by $\mathcal X$. For a triangulated category $\mathcal T$ and $\mathcal X \subseteq \mathcal T$, we still use the notation $\langle \mathcal X \rangle$ to present the smallest thick subcategory of $\mathcal T$ containing $\mathcal X$.

\begin{defn} \label{exceptional}
Let $\mathcal A$ be an abelian $k$-category and let $E \in \mathcal A$. \\
(1) We call $E$ a \emph{rigid} object if $\mathrm{Ext}_{\mathcal A}^p(E,E)=0$ for any $p \geq 1$. \\
(2) We call $E$ an \emph{exceptional} object if $E$ is rigid and $\mathrm{End}_{\mathcal A}(E)$ is a division algebra. \\
(3) A sequence of objects $(E_1,E_2, \cdots ,E_n)$ is called an \emph{exceptional sequence} if each $E_i$ is exceptional and $\mathrm{Ext}_{\mathcal A}^p(E_i,E_j)=0$ for any $i>j$ and $p \geq 0$.
\end{defn}

An abelian $k$-category $\mathcal A$ is called \emph{Hom-finite}, if for any $X,Y \in \mathcal A$, we have $\mathrm{dim}_k \mathrm{Hom}_{\mathcal A}(X,Y)< \infty$. $\mathcal A$ is called \emph{Ext-finite} if for any $X,Y \in \mathcal A$ and $i \geq 0$, we have $\mathrm{dim}_k \mathrm{Ext}_{\mathcal A}^i(X,Y)< \infty$.

Let $\mathcal A$ be an Ext-finite abelian $k$-category. For $A \in \mathcal A$, we denote by $\mathrm{add}(A)$ the full subcategory formed by direct summands of finite direct sums of $A$. A short exact sequence
\[
\epsilon :0 \longrightarrow B \longrightarrow E \longrightarrow A' \to 0
\]
in $\mathcal A$ is called a \emph{right semi-universal extension of $B$ by $A$} if $A' \in \mathrm{add}(A)$ and the connecting map
\[
\delta_{\epsilon}: \mathrm{Hom}_{\mathcal A}(A,A') \longrightarrow \mathrm{Ext}_{\mathcal A}^1(A,B)
\]
is surjective. We refer to \cite[Proposition~2.3]{chen2013universal} for the well-known existence of a semi-universal extension.

\begin{lem} \label{rigid_uni}
Let $\mathcal A$ be an Ext-finite abelian category and let $A \in \mathcal A$ be a rigid object. Assume that $0 \to B \to E \to A' \to 0$ is a right semi-universal extension of $B$ by $A$. Then $\mathrm{Ext}_{\mathcal A}^1(A,E)=0$.
\end{lem}

\begin{proof}
By definition of semi-universal extension, we have exact sequence
\[
\mathrm{Hom}_{\mathcal A}(A,A') \twoheadrightarrow \mathrm{Ext}_{\mathcal A}^1(A,B) \to \mathrm{Ext}_{\mathcal A}^1(A,E) \to \mathrm{Ext}_{\mathcal A}^1(A,A')=0.
\]
So $\mathrm{Ext}_{\mathcal A}^1(A,E)=0$.
\end{proof}

The following result is an analogue of Bongartz's completion \cite[Lemma~2.1]{Bongartz1981} with a similar proof.

\begin{prop} \label{bongartz}
Let $\mathcal H$ be an Ext-finite hereditary abelian category and let $A, B \in \mathcal H$ be rigid objects such that $\mathrm{Ext}_{\mathcal H}^1(A,B)=0$. Then there exists a rigid object $A_1 \oplus B$ with
\[
\langle A_1 \oplus B \rangle = \langle A , B \rangle
\]
\end{prop}
\begin{proof}
Choose a right semi-universal extension of $A$ by $B$
\[
\epsilon: 0 \longrightarrow A \longrightarrow A_1 \longrightarrow B' \to 0.
\]
By Lemma \ref{rigid_uni}, $\mathrm{Ext}_{\mathcal H}^1(B,A_1)=0$. Applying $\mathrm{Hom}_{\mathcal H}(-,B)$ to $\epsilon$ and using the fact that $\mathrm{Ext}_{\mathcal H}^1(A,B)=0$ and $B' \in \mathrm{add}(B)$, we obtain $\mathrm{Ext}_{\mathcal H}^1(A_1,B)=0$. Applying $\mathrm{Hom}_{\mathcal H}(A,-)$ to $\epsilon$,  we obtain $\mathrm{Ext}_{\mathcal H}^1(A,A_1)=0$. Applying $\mathrm{Hom}_{\mathcal H}(-,A_1)$ to $\epsilon$, we get $\mathrm{Ext}_{\mathcal H}^1(A_1,A_1)=0$. So the object $A_1 \oplus B$ is rigid. Clearly, we have $\langle A_1 \oplus B \rangle=\langle A , B \rangle$.
\end{proof}

The following well-known lemma is analogues to \cite[Lemma 4.1]{1982Tilted}.

\begin{lem}  \label{happel_lemma}
Let $\mathcal H$ be a Hom-finite hereditary abelian category and let $A, B \in \mathcal H$ be indecomposable objects with $\mathrm{Ext}_{\mathcal H}^1(B,A)=0$. Then each non-zero morphism $f: A \to B$ is either a monomorphism or an epimorphism.
\end{lem}

\begin{proof}
Since $\mathcal H$ is hereditary, $\mathrm{Ext}_{\mathcal H}^1(\mathrm{coker}(f),-)$ is right exact. The epimorphism $A \twoheadrightarrow \mathrm{im}(f)$ yeilds a surjective map 
\[
   \mathrm{Ext}_{\mathcal H}^1(\mathrm{coker}(f),A) \twoheadrightarrow \mathrm{Ext}_{\mathcal H}^1(\mathrm{coker}(f),\mathrm{im}(f)).
\] 
It gives rise to the following commutative diagram with exact rows.
\[
\xymatrix{
0 \ar[r] & A \ar@{.>}^{g}[r] \ar@{>>}^{p}[d] & E \ar@{.>}[r] \ar@{.>}[d] & \mathrm{coker}(f) \ar@{=}[d] \ar[r]& 0 \\
0 \ar[r] & \mathrm{im}(f) \ar^{i}[r] & B \ar[r] & \mathrm{coker}(f) \ar[r] & 0
}
\]
So we have a split short exact sequence $0 \to A \to E \oplus \mathrm{im}(f) \to B \to 0$, i.e. there exist $h:E \to A$, $q: \mathrm{im}(f) \to A$, such that $hg+qp=1_A$. Since $\mathcal H$ is Hom-finite, and $A$ is indecomposable, we have $\mathrm{End}(A)$ is a local ring. So $hg$ or $qp$ is a unit in $\mathrm{End}(A)$. That is, $g$ or $p$ is a split monomorphism. If $g$ is split mono, observing the map $\mathrm{Ext}_{\mathcal H}^1(\mathrm{coker}(f),p)$, we get $i$ is split mono. Since $B$ is indecomposable, so $i$ is an isomorphism, that is, $f$ is an epimorphism. If $p$ is split mono, then $f=ip$ is a monomorphism.
\end{proof}

\begin{cor} \label{indec_exc}
Let $\mathcal H$ be a Hom-finite hereditary abelian category. Then any indecomposable rigid object $E$ is exceptional.
\end{cor}

\begin{proof}
That is to prove $\mathrm{End}_{\mathcal H}(E)$ is a division algebra. By Lemma~\ref{happel_lemma}, any non-zero $f: E \to E$ is mono or epic. So $f$ is an isomorphism since $\mathcal H$ is Hom-finite.
\end{proof}

The following proposition seems to be known to experts. When $\mathcal H$ is a module category, the ``if" part can be deduced from \cite[Proposition~6.6]{krause2012report}, and the ``only if" part is an analogue of \cite[Corollary~4.2]{1982Tilted}.

\begin{prop} \label{exc_rigid}
Let $\mathcal H$ be an Ext-finite hereditary abelian category. Then a wide subcategory $\mathcal T$ of $\mathcal H$ can be generated by a rigid object if and only if it can be generated by an exceptional sequence.
\end{prop}

\begin{proof}
$\Rightarrow$: Assume that $\mathcal T=\langle X \rangle$ with $X$ rigid, and that $X= \oplus_{i=1}^n E_i$ such that $E_1, \cdots ,E_n$ are pairwise nonisomorphic indecomposable objects. We claim that $E_1, \cdots ,E_n$ can be reformed to an exceptional sequence. By induction, it suffices to prove that there exists $i \in \{1,2,\cdots,n \}$, such that $\mathrm{Hom}_{\mathcal H}(E_i,E_j)=0$ for any $j \neq i$.

If not, then there exists a chain of non-zero homomorphisms
\[
E_{i_1} \xrightarrow{f_1} E_{i_2} \xrightarrow{f_2} \cdots \to E_{i_r} \xrightarrow{f_r} E_{i_1}
\]
 with $i_1, \cdots , i_r$ pairwise distinct. By Lemma~\ref{happel_lemma}, each $f_i$ is either mono or epic. If all $f_i$ are mono, then the composition $f_r f_{r-1}\cdots f_1$ is mono. Since $\mathcal H$ is Hom-finite, $f_r f_{r-1}\cdots f_1$ is an isomorphism. So each $f_i$ is an isomorphism, a contradiction. Similarly, it is impossible that all $f_i$ are epimorphisms. So there exists $j\in \{1,2,\cdots,r \}$ such that $f_j$ is mono and $f_{j-1}$ is epic; here, we set $f_0=f_r$. It follows that $f_j f_{j-1} \neq 0$, so it has to be mono or epic by Lemma~\ref{happel_lemma}. If it is mono, then $f_{j-1}$ is an isomorphism, a contradiction; if it is epic, then $f_j$ is an isomorphism, also a contradiction.

$\Leftarrow$: Assume $\mathcal T= \langle E_1, E_2, \cdots ,E_n \rangle$, where $(E_1, E_2, \cdots ,E_n)$ is an exceptional sequence. By Proposition~\ref{bongartz}, there exists a rigid object $E_2'$ such that $E_1 \oplus E_2'$ is rigid and $\langle E_1 \oplus E_2' \rangle = \langle E_1, E_2\rangle$. Applying Proposition~\ref{bongartz} repeatedly, we obtain a rigid object $E_1 \oplus E_2' \oplus \cdots \oplus E_n'$ such that $\langle E_1 \oplus E_2' \oplus \cdots \oplus E_n'\rangle = \langle E_1, E_2, \cdots , E_n\rangle$.
\end{proof}

Let $\mathcal H$ be an Ext-finite hereditary abelian category. We denote by $K_0(\mathcal H)$ the \emph{Grothendieck group} of $\mathcal H$, and denote by $[X]$ the corresponding element of an object $X$. The \emph{Euler form} associated to $\mathcal H$ is by definition the bilinear form $K_0(\mathcal H) \times K_0(\mathcal H) \to \mathbb Z$ given by
\[
   \langle [X],[Y] \rangle:=\mathrm{dim}_k \mathrm{Hom}_{\mathcal H}(X,Y)-\mathrm{dim}_k \mathrm{Ext}^1_{\mathcal H}(X,Y).
\]

The following result is well known.

\begin{lem} \label{gro_independence}
   Let $(E_1, E_2, \cdots, E_n)$ be an exceptional sequence in $\mathcal H$. Then $\{ [E_1], [E_2], \cdots, [E_n] \}$ are linearly independent in $K_0(\mathcal H)$.
\end{lem}

\begin{proof}
   We observe that $\langle [E_i], [E_j] \rangle=0$ when $i>j$, and $\langle [E_i], [E_i] \rangle>0$. Then it follows directly.
\end{proof}

\section{Perpendicular calculus}
Let $\mathcal H$ be a hereditary abelian category, and let $\mathcal X \subseteq \mathcal H$ be a class of objects. The right and left \emph{perpendicular categories} of $\mathcal X$ are defined to be the full subcategories
\[
\mathcal X^{\perp} :=\{A \in \mathcal H \ | \ \mathrm{Hom}_{\mathcal H}(X,A)=0=\mathrm{Ext}_{\mathcal H}^1(X,A), \ \forall X \in \mathcal X \}
\]
and
\[
\prescript{\perp}{}{\mathcal X} :=\{A \in \mathcal H \ | \ \mathrm{Hom}_{\mathcal H}(A,X)=0=\mathrm{Ext}_{\mathcal H}^1(A,X), \ \forall X \in \mathcal X \},
\]
respectively. It is well known that both $\mathcal X^{\perp}$ and $\prescript{\perp}{}{\mathcal X}$ are wide subcategories; see \cite[Proposition~1.1]{Werner1991Perpendicular}. And then we have $\mathcal X^{\perp}=\langle \mathcal X \rangle^{\perp}$ and $\prescript{\perp}{}{\mathcal X}=\prescript{\perp}{}{\langle \mathcal X \rangle}$. Following \cite[Definition~2.1]{krause2010telescope}, we call a pair $(\mathcal X,\mathcal Y)$ of full subcategories \emph{Ext-orthogonal} if $\mathcal X^{\perp}=\mathcal Y$ and $\mathcal X=\prescript{\perp}{}{\mathcal Y}$. We call the pair \emph{complete} if each object $M$ in $\mathcal H$ fits into a five-term exact squence
\[
0 \to Y^1 \to X_0 \to M \to Y^0 \to X_1 \to 0
\]
with $X_0,X_1 \in \mathcal X$ and $Y^0,Y^1 \in \mathcal Y$. In this case, we have $\mathcal H=\langle \mathcal X, \mathcal Y \rangle$.

For a wide subcategory $\mathcal W$ of $\mathcal H$ and a class of objects $\mathcal X$ in $\mathcal W$, we use the notation $\mathcal X^{\perp_{\mathcal W}}:=\mathcal X^{\perp} \cap \mathcal W$. 

\begin{lem} \label{complete_ext_orth_pair}
   (1) Let $\mathcal W$ be a wide subcategory. If each object $M$ in $\mathcal H$ fits into a five-term exact squence
   \[
   0 \to Y^1 \to X_0 \to M \to Y^0 \to X_1 \to 0
   \]
   with $X_0,X_1 \in \mathcal W$ and $Y^0,Y^1 \in \mathcal W^\perp$. Then $(\mathcal W, \mathcal W^\perp)$ is a complete Ext-orthogonal pair.

   (2) Let $(\mathcal X, \mathcal Y)$ be a complete Ext-orthogonal pair in $\mathcal H$, and $\mathcal W$ be a wide subcategory containing $\mathcal X$. Then $(\mathcal X, \mathcal W \cap \mathcal Y)$ is a complete Ext-orthogonal pair in $\mathcal W$. 
\end{lem}

\begin{proof}
   (1) We only need to prove that $\mathcal W=\prescript{\perp}{}{(\mathcal W^{\perp})}$. Obviously $\mathcal W \subseteq \prescript{\perp}{}{(\mathcal W^{\perp})}$. It remains to prove $\prescript{\perp}{}{(\mathcal W^{\perp})} \subseteq \mathcal W$. Take $M \in \prescript{\perp}{}{(\mathcal W^{\perp})}$. Then $M$ fits into the following 5-term exact sequence
   \[
   0 \to Y^1 \to X_0 \to M \to Y^0 \to X_1 \to 0
   \]
   with $X_0,X_1 \in \mathcal W$ and $Y^0,Y^1 \in \mathcal W^{\perp}$. Since $\mathrm{Hom}_{\mathcal H}(M,Y^0)=0$, we have the following exact sequence
   \[
   0 \to Y^1 \to X_0 \to M \to 0.
   \]
   Since $\mathrm{Ext}_{\mathcal H}^1(M,Y^1)=0$, $M$ is isomorphic to a direct summand of $X_0$. So $M$ belongs to $\mathcal W$.

   (2) For any $M \in \mathcal W$, it fits into a five-term exact sequence
   \[
   0 \to Y^1 \to X_0 \to M \to Y^0 \to X_1 \to 0
   \]
   with $X_0,X_1 \in \mathcal X$ and $Y^0,Y^1 \in \mathcal Y$.
   Since $M, X_0, X_1 \in \mathcal W$, we have that $Y^0, Y^1 \in \mathcal W \cap \mathcal Y$. Obviously $\mathcal W \cap \mathcal Y=\mathcal X^{\perp_{\mathcal W}}$. So we have done by (1).
\end{proof}

Let $\mathcal A$ be an abelian category and let $X, Y \in \mathcal A$. A \emph{right $\mathrm{add}(X)$-approximation of $Y$} is a morphism $f: X' \to Y$ with $X' \in \mathrm{add}(X)$, such that $\mathrm{Hom}_{\mathcal A}(X,f): \mathrm{Hom}_{\mathcal A}(X,X') \to \mathrm{Hom}_{\mathcal A}(X,Y)$ is surjective. Such an approximation always exists if $\mathcal A$ is Hom-finite over a field.

\begin{lem} \label{appr}
Let $\mathcal A$ be an abelian category and let $X \in \mathcal A$ be a rigid object. Assume that
\[
0 \to K \to X_E \xrightarrow{f} E \to Y \to 0
\]
is an exact sequence with $f$ a right $\mathrm{add}(X)$-approximation of $E$. Then we have $\mathrm{Ext}_{\mathcal A}^1(X,K)=0$. Moreover, if $\mathcal A$ is hereditary, then $\mathrm{Hom}_{\mathcal A}(X,Y)=0$.
\end{lem}

\begin{proof}
Let
\[
\xymatrix{ X_E \ar@{->>}^{i}[r] & \mathrm{Im} f \ar@{^(->}^{c}[r] & E
}
\]
be an epi-mono factorization of $f$. Then $\mathrm{Hom}_{\mathcal A}(X,f)$ has the following factorization.
\[
\mathrm{Hom}_{\mathcal A}(X, X_E) \xrightarrow{\mathrm{Hom}_{\mathcal A}(X,i)} \mathrm{Hom}_{\mathcal A}(X,\mathrm{Im} f) \xrightarrow{\mathrm{Hom}_{\mathcal A}(X,c)} \mathrm{Hom}_{\mathcal A}(X,E).
\]
Clearly, the map $\mathrm{Hom}_{\mathcal A}(X,c)$ is injective. Since $f$ is an approximation, $\mathrm{Hom}_{\mathcal A}(X, f)$ is surjective. So $\mathrm{Hom}_{\mathcal A}(X,c)$ is bijective, and $\mathrm{Hom}_{\mathcal A}(X,i)$ is surjective.

Applying $\mathrm{Hom}_{\mathcal A}(X,-)$ to  $0 \to K \to X_E \to \mathrm{Im}f \to 0$  and using the fact that $\mathrm{Ext}_{\mathcal A}^1(X, X_E)=0$, we get $\mathrm{Ext}_{\mathcal A}^1(X,K)=0$.

If $\mathcal A$ is hereditary, then $\mathrm{Ext}_{\mathcal A}^1 (X,-)$ is right exact. By $\mathrm{Ext}_{\mathcal A}^1 (X,X)=0$, we have $\mathrm{Ext}_{\mathcal A}^1(X,\mathrm{Im}f)=0$. Applying $\mathrm{Hom}_{\mathcal A}(X,-)$ to $0 \to \mathrm{Im}f \to E \to Y \to 0$, we obtain the following exact sequence
\[
0 \to \mathrm{Hom}_{\mathcal A}(X,\mathrm{Im}f) \xrightarrow{\sim} \mathrm{Hom}_{\mathcal A}(X,E) \to \mathrm{Hom}_{\mathcal A}(X,Y) \to \mathrm{Ext}_{\mathcal A}^1(X,\mathrm{Im}f)=0.
\]
So we have $\mathrm{Hom}_{\mathcal A}(X,Y)=0$.
\end{proof}

The perpendicular calculus for the module category over a finite dimensional hereditary algebra is treated in \cite[Proposition A.1]{hubery2016categorification}. The next proposition is a natural generalization of this to a hereditary abelian category, and might be viewed as a finite version of \cite[Theorem~2.2]{krause2010telescope}. The proof given bellow is similar to the one of \cite[Proposition~3.2]{Werner1991Perpendicular}.

\begin{prop} \label{perp_cal}
Let $\mathcal H$ be an Ext-finite hereditary abelian category and let $X \in \mathcal H$ be a rigid object. Then the following statements hold. 

(1) Any $M \in \mathcal H$ fits into a 5-term exact sequence
\[
0 \to Y^1 \to X_0 \to M \to Y^0 \to X_1 \to 0
\]
with $X_0,X_1 \in \langle X \rangle$ and $Y^0,Y^1 \in X^{\perp}$. 

(2) The pair $(\langle X \rangle,X^{\perp})$ is Ext-orthogonal.
\end{prop}

\begin{proof}
(1) Take a right semi-universal extension of $M$ by $X$,
\[
0 \to M \to E \to X^M \to 0
\]
with $X^M \in \mathrm{add}(X)$. By Lemma \ref{rigid_uni}, $\mathrm{Ext}_{\mathcal H}^1(X,E)=0$. We take a right add($X$)-approximation $X_E \to E$ of $E$, and get an exact sequence
\[
0 \to K \to X_E \to E \to Y^0 \to 0.
\]
By Lemma \ref{appr}, we have $\mathrm{Ext}_{\mathcal H}^1(X,K)=0$ and $\mathrm{Hom}_{\mathcal H}(X,Y^0)=0$. Write $J$ for the image of the composition $X_E \to E \to X^M$, and write $L$ for its kernel. We have $L \in \langle X \rangle$ and $J \in \langle X \rangle$. Now we obtain an exact commutative diagram bellow.
\[
\xymatrix{
0 \ar[r] & L \ar[r] \ar@{.>}[d] & X_E \ar[r] \ar[d] & J \ar[r] \ar@{^(->}[d] & 0 \\
0 \ar[r] & M \ar[r] \ar@{.>}[d] & E \ar[r] \ar@{.>}[d] & X^M \ar[r] \ar@{.>}[d] & 0 \\
0 \ar[r] & N \ar@{.>}[r] \ar[d] & Y^0 \ar@{.>}[r] \ar[d] & X_1 \ar[r] \ar[d] & 0 \\
         & 0               & 0               & 0
}
\]
Here, the third row is defined to be the cokernel of the previous two rows, and it is exact by the snake lemma.
Since $\mathrm{Ext}_{\mathcal H}^1(X,E)=0$, we have $\mathrm{Ext}_{\mathcal H}^1(X,Y^0)=0$, and then $Y^0 \in X^{\perp}$. We observe that $X_1 \in \langle X \rangle$ from the rightmost column. Writing $I$ for the image of $L \to M$, we obtain an exact sequence
\begin{align} \label{exa1}
0 \to I \to M \to Y^0 \to X_1 \to 0
\end{align}
with $Y^0 \in X^{\perp}$, $X_1 \in \langle X \rangle$.

By the snake lemma, $K$ is also the kernel of $L \to M$. Take a right add($X$)-approximation $X_K \to K$ of $K$, and write $Y^1$ for its cokernel. By a similar proof to the fact that $Y^0 \in X^{\perp}$, we have $Y^1 \in X^{\perp}$. We have an exact commutative diagram
\[
\xymatrix{
0 \ar[r] & X_K \ar@{=}[r] \ar[d] & X_K \ar[r] \ar[d] & 0 \ar[r] \ar[d] & 0 \\
0 \ar[r] & K \ar[r] \ar@{.>}[d] & L \ar[r] \ar@{.>}[d] & I \ar[r] \ar@{:}[d] & 0 \\
0 \ar[r] & Y^1 \ar@{.>}[r] \ar[d] & X_0 \ar@{.>}[r] \ar[d] & I \ar[r] \ar[d] & 0 \\
 & 0 & 0 & 0 &
}
\]
Here, the third row is defined to be the cokernel of the previous two rows. By the snake lemma, $Y^1 \to X_0$ is mono. Since $X_K, L \in \langle X \rangle$, we have $X_0 \in \langle X \rangle$. Gluing the third row and (\ref{exa1}), we obtain the required 5-term exact sequence.

(2) Since a perpendicular category is wide, we have 
\[
   M \in X^\perp \Leftrightarrow X \in \prescript{\perp}{}{M} \Leftrightarrow \langle X \rangle \in \prescript{\perp}{}{M} \Leftrightarrow M \in \langle X \rangle ^\perp.
\]
So $X^\perp=\langle X \rangle ^\perp$. Then it follows from Lemma~\ref{complete_ext_orth_pair}(2).
\end{proof}

\begin{rmk}
   {\rm One can give another proof of the proposition using \cite[Theorem~3.2]{bondal1989representation} and Proposition~\ref{exc_rigid}.}
\end{rmk}

\begin{cor} \label{wide_gen_exc}
   Let $\mathcal H$ be an Ext-finite hereditary abelian category satisfying $\mathrm{rank}(K_0(\mathcal H))< \infty$, and let $\mathcal W$ be a wide subcategory of $\mathcal H$. If each indecomposable object in $\mathcal W$ is exceptional, then $\mathcal W$ can be generated by an exceptional sequence.
\end{cor}

\begin{proof}
   Take an indecomposable object $E_1 \in \mathcal W$. If $E_1^{\perp_{\mathcal W}} \neq 0$, then we can take an indecomposable object $E_2 \in E_1^{\perp_{\mathcal W}}$. Repeating this prograss, we will get an exceptional sequence $(E_1, E_2,\cdots)$ in $\mathcal W$. Since $\mathrm{rank}(K_0(\mathcal H))< \infty$, the prograss will end in finite steps by Lemma~\ref{gro_independence}. Finally we get an exceptional sequence $(E_1, E_2,\cdots, E_n)$ such that $(E_1, E_2,\cdots, E_n)^{\perp_{\mathcal W}}=0$. By Proposition~\ref{perp_cal}(1), we have $\mathcal W=\langle E_1, E_2,\cdots, E_n \rangle$.
\end{proof}

We might compare the condition of Proposition~\ref{perp_cal} and  \cite[Proposition~A.1]{hubery2016categorification} via the next proposition.

\begin{prop} \label{proj_gen}
   Let $\mathcal H$ be a Hom-finite hereditary abelian category and $P$ a projective object of $\mathcal H$. Then $\mathcal H=\langle P \rangle$ if and only if $P$ is a generator of $\mathcal H$, i.e. for any $X \in \mathcal H$, there exists $n \in \mathbb N$ and an epimorphism $P^n \twoheadrightarrow X$.
\end{prop}
\begin{proof}
   $\Rightarrow$: For any $X \in \mathcal H$, we take $f:P'\to M$ to be a right add($P$)-approximation of $X$. By Lemma~\ref{appr}, $\mathrm{Hom}_{\mathcal H}(P,\mathrm{coker}f)=0$. Since $P$ is projective, we have $\mathrm{Ext}^1_{\mathcal H}(P,\mathrm{coker}f)=0$. So $\mathrm{coker}f \in P^{\perp}=\mathcal H^{\perp}=0$, i.e. $f$ is epi. 

   $\Leftarrow:$ Since $P$ is a generator of $\mathcal H$, for any $X \in \mathcal H$, we have a presensentation $P^m \to P^n \to X \to 0$. So $X \in \langle P \rangle$.
\end{proof}

%Dually, we have
%
%\begin{prop}
%Let $\mathcal H$ be an Ext-finite hereditary abelian $k$-category and let $Y \in \mathcal H$ be a rigid object. Then the following statements hold. \\
%(1) Any $M \in \mathcal H$ fits into a functorial 5-term exact sequence
%\[
%0 \to Y^1 \to X_0 \to M \to Y^0 \to X_1 \to 0
%\]
%with $X_0,X_1 \in \prescript{\perp}{}{Y}$ and $Y^0,Y^1 \in \langle Y \rangle$. \\
%(2) The pair $(\prescript{\perp}{}{Y},\langle Y \rangle )$ is Ext-orthogonal.
%\end{prop}

\section{Uniserial categories with Serre duality}

Let $\mathcal U$ be a \emph{length abelian category}, that is, an abelian category whose objects have finite length. Let $E$ be an object, we denoted by $l(E)$ the length of $E$. An object $U$ is called \emph{uniserial} if it has a unique composition series. If all indecomposable objects in $\mathcal U$ are uniserial, then we call $\mathcal U$ a \emph{uniserial category}.

Let $\mathcal U$ be a uniserial category and $S$ a simple object in $\mathcal U$. Up to isomorphism, we denote by $S_{[n]}$ the unique indecomposable object with socle $S$ and length $n$ (such $S_{[n]}$ may not exist). Then $S_{[n]}$ has the unique chain of all subobjects:
\[
0 \subseteq S \subseteq S_{[2]} \subseteq \cdots \subseteq S_{[n-1]} \subseteq S_{[n]}.
\]
And each indecomposable object in $\mathcal U$ is of the form $S_{[n]}$; see \cite[Subsection~1.7]{chen2009introduction}. Dually, we denote by $S^{[n]}$ the unique indecomposable object with top $S$ and length $n$.

For an object $E$ in a length category, we denote by $\mathrm{top}(E)$ and $\mathrm{soc}(E)$ the top and the socle of $E$, respectively. We denote by $\mathrm{CF}(E)$ the set of isomorphism class of the composition factors of $E$. For a subcategory $\mathcal X$, we denote by $\mathrm{CF}(\mathcal X):= \cup_{X \in \mathcal X} \mathrm{CF}(X)$. 

\begin{lem} \label{uni_hom}
   Let $\mathcal U$ be a uniserial category and $E$, $F$ indecomposable objects. Then $\mathrm{Hom}_{\mathcal U}(E,F) \neq 0$ if and only if $\mathrm{top}(E) \in \mathrm{CF}(F)$ and $\mathrm{soc}(F) \in \mathrm{CF}(E)$.
\end{lem}
   
\begin{proof}
   If $\mathrm{Hom}_{\mathcal U}(E,F) \neq 0$, we take a nonzero morphism $f: E \to F$. We have $\mathrm{top}(\mathrm{im} f) \cong \mathrm{top}(E)$ and $\mathrm{soc}(\mathrm{im} f) \cong \mathrm{soc}(F)$. So $\mathrm{top}(E) \in \mathrm{CF}(F)$ and $\mathrm{soc}(F) \in \mathrm{CF}(E)$. 
   
   Conversely, since $\mathrm{soc}(F) \in \mathrm{CF}(E)$, there exists a quotient object $Q$ of $E$ whose socle is isomorphic to $\mathrm{soc}(F)$. So there exists a shortest object $I$ in $\mathcal U$ such that $\mathrm{top}(I)\cong \mathrm{top}(E)$ and $\mathrm{soc}(I)\cong \mathrm{soc}(F)$. So $I$ is a quotient of $Q$ since $\mathrm{top}(Q) \cong \mathrm{top}(I)$ and $l(Q) \geq l(I)$. Then $I$ is a quotient of $E$. 
   
   Since $\mathrm{top}(E) \in \mathrm{CF}(F)$, there exists a subobject $Q'$ of $F$ with $\mathrm{top}(Q') \cong \mathrm{top}(E)$. By the minimality of $l(I)$, we have $l(I) \leq l(Q') \leq l(F)$. Then $\mathrm{soc}(I) \cong \mathrm{soc}(F)$ and $l(I) \leq l(F)$ implies that $I$ is a subobject of $F$.

   Now we can get a composition of an epimorphism and a monomorphism $E \twoheadrightarrow I \hookrightarrow F$, which implies that $\mathrm{Hom}_{\mathcal U}(E,F) \neq 0$.
\end{proof}

Let $\mathcal H$ be a Hom-finite abelian category. We say that $\mathcal H$ satisfies \emph{Serre duality} if there exist an equivalence $\tau: \mathcal H \to \mathcal H$ and natural isomorphisms
\[
\mathrm{Ext}_{\mathcal H}^1(X,Y) \xrightarrow{\sim} \mathrm{D \ Hom}_{\mathcal H}(Y, \tau X)
\]
for all objects $X,Y$ from $\mathcal H$. Here, $\mathrm{D}=\mathrm{Hom}_k(-,k)$ denotes the $k$-duality. It follows from \cite[Proposition~1.4]{lenzing2006hereditary} that $\mathcal H$ is hereditary without nozero projectives or injectives. It is well known that a Hom-finite length category with Serre duality is uniserial. We refer to \cite[Theorem 1.7]{lenzing2006hereditary} and \cite[Proposition 8.3]{gabriel1973indecomposable}.

Recall that an additive category $\mathcal A$ is \emph{connected}, if any decomposition $\mathcal A=\mathcal A_1 \coprod \mathcal A_2$ into additive categories implies that $\mathcal A_1=0$ or $\mathcal A_2=0$. 

In what follows, we denote by $\mathcal U_n$ a connected uniserial hereditary abelian category with Serre duality which has exactly $n$ simples up to isomorphism. According to \cite[Subsection~1.8]{chen2009introduction}, we can assume that $\{S_1,\cdots,S_n \}$ is a complete set of representatives of the isomorphic classes of simple objects in $\mathcal U_n$, and $\tau S_i=S_{i-1}$, for any $i=1,2, \cdots n$. Here, we set $S_0:=S_n$.

\begin{lem} \label{uni_exc}
   Let $E \in \mathcal U_n$ be an indecomposable object. Then $E$ is exceptional if and only if $l(E) < n$.
\end{lem}

\begin{proof}
   By Corollary~\ref{indec_exc}, $E$ is exceptional if and only if $\mathrm{Ext}_{\mathcal U_n}^1(E,E)=0$, by Serre duality, if and only if $\mathrm{Hom}_{\mathcal U_n}(E, \tau E) = 0$. 
   
   By Lemma~\ref{uni_hom}, we have that $\mathrm{Hom}_{\mathcal U_n}(E, \tau E) = 0$ if and only if $l(E) < n$.
\end{proof}

\begin{lem} \label{wide_length}
   Let $\mathcal W$ be a wide subcategory of $\mathcal U_n$ and let $U \in \mathcal W$ be an indecomposable object with $l(U)>n$. Then there exists an indecomposable object $V \in \mathcal W$ such that $l(V)=n$.
\end{lem}

\begin{proof}
   Denote $S:=\mathrm{soc}(U)$ and $l(U)=m>n$. We have an exact sequence 
   \[
   0 \to S_{[n]} \to U \xrightarrow{p} S_{[m-n]} \to 0, 
   \]
   and a monomorphism $S_{[m-n]} \xrightarrow{i} U$. Now we obtain a left exact sequence 
   \[
      0 \to S_{[n]} \to U \xrightarrow{i \circ p} U.
   \]
   Since $\mathcal W$ is wide, we have $S_{[n]} \in \mathcal W$.
\end{proof}

For a hereditary abelian category $\mathcal H$, we denote by $\mathrm{Wid}(\mathcal H)$ the set of wide subcategories of $\mathcal H$, and $\mathrm{Exc}(\mathcal H)$ the set of wide subcategories which are generated by an exceptional sequence, and $\mathrm{NExc}(\mathcal H):=\mathrm{Wid}(\mathcal H) \setminus \mathrm{Exc}(\mathcal H)$.

The following proposition is essentially due to \cite[Lemma~2.3.1]{dichev2009thick}. Here, we provide a different proof.

\begin{prop} \label{uniserial}
   The following statements hold.

(1) Let $T \in \mathcal U_n$ be a rigid object. Then $\mathrm{CF}(T) \neq \{S_1,\cdots,S_n \}$.

(2) Let $\mathcal W \in \mathrm{Wid}(\mathcal U_n)$. Then $\mathcal W \in \mathrm{Exc}(\mathcal U_n)$ if and only if 
\[
   \mathrm{CF}(\mathcal W) \neq \{S_1,\cdots,S_n \}.
\]
In particular, $\mathcal U_n \in \mathrm{NExc}(\mathcal U_n)$.

(3) Let $\mathcal W \in \mathrm{Wid}(\mathcal U_n)$. Then $\mathcal W \in \mathrm{NExc}(\mathcal U_n)$ if and only if there exists an indecomposable object $U \in \mathcal W$ such that $l(U)=n$.
\end{prop}

\begin{proof}
(1) If $\mathrm{CF}(T) = \{S_1,\cdots,S_n \}$, we take a longest indecomposable direct summand $E$ of $T$. By Corollary~\ref{indec_exc}, $E$ is exceptional. And by Lemma~\ref{uni_exc}, $l(E)<n$, so $\mathrm{top}(\tau^{-}E) \notin \mathrm{CF}(E)$. Since $\mathrm{CF}(T) = \{S_1,\cdots,S_n \}$, there exists an indecomposable direct summand $F$ of $T$ such that $\mathrm{top}(\tau^{-}E) \in \mathrm{CF}(F)$. Then $\mathrm{soc}(F) \in \mathrm{CF}(\tau^- E)$ since $l(F) \leq l(\tau^- E)$. By Lemma~\ref{uni_hom}, $\mathrm{Hom}_{\mathcal U_n}(\tau ^{-}E, F) \neq 0$. It follows that $\mathrm{Ext}_{\mathcal U_n}^1(F,E) \neq 0$ by Serre duality. A contradiction.

(2) $\Rightarrow$: By Proposition~\ref{exc_rigid}, there exists a rigid object $T$ such that $\mathcal W= \langle T \rangle$. So $\mathrm{CF}(\mathcal W)=\mathrm{CF}(T)$. Then we are done by (1).

$\Leftarrow$: Since $\mathrm{CF}(\mathcal W) \neq \{S_1,\cdots,S_n \}$, each indecomposable object $E$ in $\mathcal W$ satisfies $l(E)<n$, so $E$ is exceptional by Lemma~\ref{uni_exc}. Notice that $K_0(\mathcal U_n)=\mathrm{Span}_{\mathbb Z}([S_1], \cdots, [S_n])$. By Lemma~\ref{wide_gen_exc}, we have $\mathcal W \in \mathrm{Exc}(\mathcal U_n)$.

(3) $\Rightarrow$: If not, by Lemma~\ref{wide_length}, each indecomposable object in $\mathcal W$ has length strictly less than $n$, hence is exceptional by Lemma~\ref{uni_exc}. By Lemma~\ref{wide_gen_exc}, we have $\mathcal W \in \mathrm{Exc}(\mathcal U_n)$. A contradiction.

$\Leftarrow$: If there is an indecomposable object $U$ with $l(U)=n$, then clearly $\mathrm{CF}(\mathcal W) = \{S_1,\cdots,S_n \}$. By (2), $\mathcal W \in \mathrm{NExc}(\mathcal U_n)$.
\end{proof}

\begin{lem} \label{U0_perp}
   The following statements hold.

   (1) $\prescript{\perp}{}{S^{[n]}}=\langle S, \tau S, \cdots , \tau^{n-2}S \rangle \in \mathrm{Exc}(\mathcal U_n)$.

   (2) $\langle S, \tau S, \cdots , \tau^{n-2}S \rangle^\perp=\langle S^{[n]} \rangle \in \mathrm{NExc}(\mathcal U_n)$.

   (3) $(\prescript{\perp}{}{S^{[n]}},\langle S^{[n]} \rangle)$ is a complete orthogonal pair in $\mathcal U_n$.
\end{lem}

\begin{proof} 
   (1) Let $M$ be an indecomposable objet. By Lemma~\ref{uni_hom}, we have that $\mathrm{Hom}_{\mathcal U_n}(M, S^{[n]})=0$ if and only if $\mathrm{soc}(S^{[n]}) \notin \mathrm{CF}(M)$, i.e. $\tau^{n-1}S \notin \mathrm{CF}(M)$. By Serre duality and Lemma~\ref{uni_hom}, $\mathrm{Ext}^1_{\mathcal U_n}(M, S^{[n]})=0$ if and only if $\mathrm{top}(\tau^-S^{[n]}) \notin \mathrm{CF}(M)$, i.e. $\tau^{n-1}S \notin \mathrm{CF}(M)$. So $M \in \prescript{\perp}{}{S^{[n]}}$ if and only if $\tau^{n-1}S \notin \mathrm{CF}(M)$. And $\tau^{n-1}S \notin \mathrm{CF}(M)$ if and only if $M \cong \tau^i(S^{[j]})$, where $1 \leq i \leq n-2$, $1 \leq j \leq n-1-i$, if and only if $M \in \langle S, \tau S, \cdots , \tau^{n-2}S \rangle$. In a word, $M \in \prescript{\perp}{}{S^{[n]}}$ if and only if $M \in \langle S, \tau S, \cdots , \tau^{n-2}S \rangle$.

   Obviously $(S, \tau S, \cdots , \tau^{n-2}S)$ is an exceptional sequence.

   (2) Let $M$ be an indecomposable object. Then $\mathrm{Hom}_{\mathcal U_n}(\tau^i S, M)=0$ for $0 \leq i \leq n-2$ if and only if $\mathrm{soc}(M) \cong \tau^{n-1}S$. And $\mathrm{Ext}_{\mathcal U_n}(\tau^i S, M)=0$ for $0 \leq i \leq n-2$ if and only if $\mathrm{top}(M) \cong S$. So $M \in \langle S, \tau S, \cdots , \tau^{n-2}S \rangle^\perp$ if and only if $M \cong S^{[ln]}$ for some $l \geq 1$. So
   \[
      \mathrm{add}\{S^{[ln]}\ | \ l \in \mathbb N \}=\langle S, \tau S, \cdots , \tau^{n-2}S \rangle^\perp
   \]
   is a wide subcategory. It follows that$\langle S^{[n]} \rangle \subseteq \mathrm{add}\{S^{[ln]}\ | \ l \in \mathbb N \}$. And clearly $\langle S^{[n]} \rangle \supseteq \mathrm{add}\{S^{[ln]}\ | \ l \in \mathbb N \}$.
   So 
   \[
      \langle S^{[n]} \rangle = \mathrm{add}\{S^{[ln]}\ | \ l \in \mathbb N \}.
   \]
   Then $M \in \langle S, \tau S, \cdots , \tau^{n-2}S \rangle^\perp$ if and only if $M \in \langle S^{[n]} \rangle$. By Proposition~\ref{uniserial}(3), we have $\langle S^{[n]} \rangle \in \mathrm{NExc}(\mathcal U_n)$.

   (3) By (1), we have $\prescript{\perp}{}{S^{[n]}} \in \mathrm{Exc}(\mathcal U_n)$. By (2) and Proposition~\ref{perp_cal}(2), $(\prescript{\perp}{}{S^{[n]}},\langle S^{[n]} \rangle)$ is a complete Ext-orthogonal pair in $\mathcal U_n$.

\end{proof}

The following result is essentially due to \cite[Theorem~2.3.25]{dichev2009thick}, whose proof seems not rigorous. Here, we provide a proof.

\begin{prop}  \label{thick_uniserial}
We have a mutually invertible pair of one-to-one correspondences

\begin{align*}
\mathrm{Exc}(\mathcal U_n) &\leftrightarrow \mathrm{NExc}(\mathcal U_n)  \\
                        \mathcal W  &\mapsto \mathcal W^ \perp \\
   \prescript{\perp}{}{\mathcal V}  &\mapsfrom \mathcal V
\end{align*}
%Dually, we also have a mutually invertible pair of one-to-one correspondences \\
%\begin{align*}
%\mathrm{Exc}(\mathcal U_n) &\leftrightarrow \mathrm{NExc}(\mathcal U_n) \\
%   \mathcal W  &\mapsto \prescript{\perp}{}{\mathcal W} \\
%   {\mathcal V}^{\perp}  &\mapsfrom \mathcal V
%\end{align*}
\end{prop}

\begin{proof}
For any $\mathcal W \in \mathrm{Exc}(\mathcal U_n)$, we claim that $\mathcal W^{\perp} \in \mathrm{NExc}(\mathcal U_n)$. If not, suppose that $\mathcal W = \langle E_1, \cdots , E_r \rangle$, $\mathcal W^\perp = \langle F_1, \cdots , F_s \rangle$, where $(E_1, \cdots , E_r)$ and $(F_1, \cdots , F_s)$ are exceptional sequences. Then $(E_1, \cdots , E_r, F_1, \cdots , F_s)$ is an exceptional sequence. By Proposition~\ref{perp_cal}(2), 
\[
   \mathcal U_n = \langle E_1, \cdots , E_r, F_1, \cdots , F_s \rangle \in \mathrm{Exc}(\mathcal U_n),
\]
which contradicts to Proposition~\ref{uniserial}(2).

On the other hand, for any $\mathcal V \in \mathrm{NExc}(\mathcal U_n)$, we claim that $\prescript{\perp}{}{\mathcal V} \in \mathrm{Exc}(\mathcal U_n)$. If not, by Proposition~\ref{uniserial}(3), there exists an indecomposable object $U \in \prescript{\perp}{}{\mathcal V}$ with $l(U)=n$. And there exists an indecomposable object $V \in \mathcal V$ with $l(V)=n$, so $\mathrm{Hom}_{\mathcal U_n}(U,V) \neq 0$ by Lemma~\ref{uni_hom}. A contradiction.

So the maps are well defined. By Proposition~\ref{perp_cal}(2), we have $\mathcal W=\prescript{\perp}{}{(\mathcal W^\perp)}$ for $\mathcal W \in \mathrm{Exc}(\mathcal U_n)$. It remains to prove $\mathcal V=(\prescript{\perp}{}{V})^\perp$ for $\mathcal V \in \mathrm{NExc}(\mathcal U_n)$. 

By Proposition~\ref{uniserial}(3), there is an indecomposable object $U_0 \in \mathcal V$ with $l(U_0)=n$. By Lemma~\ref{U0_perp}(3), we have that $(\prescript{\perp}{}{U_0}, \langle U_0 \rangle)$ is a complete Ext-orthogonal pair in $\mathcal U_n$. By the dual version of Lemma~\ref{complete_ext_orth_pair}(2), $(\mathcal V \cap \prescript{\perp}{}{U_0}, \langle U_0 \rangle)$ is a complete Ext-orthogonal pair in $\mathcal V$, hence $\mathcal V=\langle \mathcal V \cap \prescript{\perp}{}{U_0}, U_0 \rangle$. For the same reason, $((\prescript{\perp}{}{\mathcal V})^\perp \cap \prescript{\perp}{}{U_0}, \langle U_0 \rangle)$ is a complete Ext-orthogonal pair in $(\prescript{\perp}{}{\mathcal V})^\perp$. 

We have 
\begin{gather} \label{U_0}
   \prescript{\perp}{}{U_0} \cap \prescript{\perp}{}{(\prescript{\perp}{}{U_0}\cap \mathcal V)}=\prescript{\perp}{}{\langle U_0, \prescript{\perp}{}{U_0} \cap \mathcal V \rangle}=\prescript{\perp}{}{\mathcal V} \subseteq \prescript{\perp}{}{U_0}.
\end{gather}
Recall from the proof of Lemma~\ref{U0_perp}(1) that each indecomposable element in $\prescript{\perp}{}{U_0}$ is exceptional. By Corollary~\ref{wide_gen_exc}, we have $\prescript{\perp}{}{U_0} \cap \mathcal V \in \mathrm{Exc}(\prescript{\perp}{}{U_0})$, and $\prescript{\perp}{}{\mathcal V} \in \mathrm{Exc}(\mathcal U_n)$. By the dual version of Proposition~\ref{perp_cal}(2) and \ref{U_0}, $(\prescript{\perp}{}{\mathcal V}, \prescript{\perp}{}{U_0} \cap \mathcal V)$ is a complete Ext-orthogonal pair in $\prescript{\perp}{}{U_0}$. Since $\prescript{\perp}{}{\mathcal V} \in \mathrm{Exc}(\mathcal U_n)$, by Proposition~\ref{perp_cal}(2), $(\prescript{\perp}{}{\mathcal V},(\prescript{\perp}{}{\mathcal V})^\perp)$ is a complete Ext-orthogonal pair in $\mathcal U_n$. And then by Lemma~\ref{complete_ext_orth_pair}(2), $(\prescript{\perp}{}{\mathcal V}, \prescript{\perp}{}{U_0} \cap (\prescript{\perp}{}{\mathcal V})^\perp)$ is also a complete Ext-orthogonal pair in $\prescript{\perp}{}{U_0}$. Now we have proved that 
\[
   \mathcal V \cap \prescript{\perp}{}{U_0}=(\prescript{\perp}{}{\mathcal V})^\perp \cap \prescript{\perp}{}{U_0}.
\]
So we have 
\[
   \mathcal V= \langle \mathcal V \cap \prescript{\perp}{}{U_0}, U_0 \rangle=\langle (\prescript{\perp}{}{\mathcal V})^\perp \cap \prescript{\perp}{}{U_0}, U_0 \rangle=(\prescript{\perp}{}{\mathcal V})^\perp.
\]
The last equality uses the fact that $((\prescript{\perp}{}{\mathcal V})^\perp \cap \prescript{\perp}{}{U_0}, \langle U_0 \rangle)$ is a complete Ext-orthogonal pair in $(\prescript{\perp}{}{\mathcal V})^\perp$. The proof is complete.
\end{proof}

\begin{rmk}
{\rm We view $\mathrm{Wid}(\mathcal U_n)$ as a poset ordered by inclusion. It is known that $\mathrm{Wid}(\mathcal U_n)$ is isomorphic to $\mathrm{NC}^B(n)$ as a lattice, where $\mathrm{NC}^B(n)$ is the noncrossing partition of type $B$. We refer to \cite[Corollary 9.2]{krause2021category}. By the way, the poset $\mathrm{Exc}(\mathcal U_n)$ is not a lattice:  $\langle S_1 \rangle$ and $\langle S_2,\cdots,S_n \rangle \in \mathrm{Exc}(\mathcal U_n)$, but $\langle S_1,\cdots,S_n \rangle=\mathcal U_n \in \mathrm{NExc}(\mathcal U_n)$. }
\end{rmk}

The following proposition about the structure of perpendicular category in $\mathcal U_n$ is well known.

\begin{prop} \label{uni_exc_perp}
   Let $S \in \mathcal U_n$ be a simple object, and let $m<n$. Then 
   \[
      (S^{[m]})^\perp=(S, \tau S, \cdots , \tau^{m-1} S)^\perp \coprod \langle \tau S, \tau^2 S, \cdots , \tau^{m-1} S \rangle.
   \]
\end{prop}

\begin{proof}
   Let $U \in (S^{[m]})^\perp$ be an indecomposable object not in $\langle \tau S, \cdots , \tau^{m-1} S \rangle$. Then there exists $r \in \{m, m+1, \cdots , n \}$ such that $\tau^r S \in \mathrm{CF}(U)$. We want to prove that $U \in (S, \tau S, \cdots , \tau^{m-1} S)^\perp$.
   
   We claim that $\mathrm{soc}(U) \notin \{S, \tau S, \cdots , \tau^{m-1} S \}$. Or else $\tau^r S \in \mathrm{CF}(U)$ implies that $S \in \mathrm{CF}(U)$. It contradicts to $\mathrm{Hom}_{\mathcal U_n}(S^{[m]},U)=0$.

   We claim that $\mathrm{top}(U) \notin \{\tau S, \tau^2 S, \cdots , \tau^{m} S \}$. Or else $\tau^r S \in \mathrm{CF}(U)$ implies that $\tau^m S \in \mathrm{CF}(U)$. Then by Lemma~\ref{uni_hom}, we have $\mathrm{Hom}_{\mathcal U_n}(U, \tau S^{[m]}) \neq 0$. By Serre duality, it contradicts to $\mathrm{Ext}_{\mathcal U_n}^1(S^{[m]},U) = 0$.

   So $U \notin \langle \tau S, \tau^2 S, \cdots , \tau^m S \rangle$ implies that 
   \begin{gather} \label{soc_U_not_in}
      \mathrm{soc}(U) \notin \{S, \tau S, \cdots , \tau^{m-1} S \} \text{ and } \mathrm{top}(U) \notin \{\tau S,\tau^2 S, \cdots , \tau^{m} S \}.
   \end{gather}
   By Lemma~\ref{uni_hom} and Serre duality, property (\ref{soc_U_not_in}) holds if and only if 
   \[
      U \in (S, \tau S, \cdots , \tau^{m-1} S)^\perp.
   \]

   It remains to prove that for any indecomposable $X \in (S, \tau S, \cdots , \tau^{m-1} S)^\perp$, $Y \in \langle \tau S, \cdots , \tau^{m-1} S \rangle$, we have $\mathrm{Hom}_{\mathcal U_n}(X,Y)=0=\mathrm{Hom}_{\mathcal U_n}(Y,X)$. Notice that $Y \cong \tau^i S^{[j]} $ for some $1 \leq i \leq m-1$, $1 \leq j \leq m-i$. By Lemma~\ref{uni_hom}, the proposition holds.
\end{proof}

\begin{rmk}
   { \rm Let $k$ be an algebraically closed field and $\Gamma _n$ the quiver of $\tilde{\mathbb A}_n$ type with cyclic orientation. Then the $k$-category $\mathcal U_n$ is equivalent to $\mathrm{rep}_0(\Gamma_n, k)$, the category of nilpotent representations of $\Gamma_n$. Let $E$ be an exceptional object of length $m$ in $\mathcal U_n$, then Proposition~\ref{uni_exc_perp} implies that 
   \[
   E^\perp \simeq \mathrm{rep}_0(\Gamma_{n-m}, k) \coprod \mathrm{rep}(\vec{\mathbb A}_{m-1}, k),
   \] 
   where $\vec{\mathbb A}_{m-1}$ means the equioriented $\mathbb A_{m-1}$ quiver. }
\end{rmk}

\section{Weighted projective lines}

We recall the definition of a weighted projective line and coherent sheaves, which is due to \cite[Section~1]{geigle1987class}. 

Let $\bm p=(p_1,p_2,\cdots,p_n)\in \mathbb N_+^n$. We define an abelian group
\[
L(\bm p)=\langle\vec{x_1},\vec{x_2},\cdots,\vec{x_n} \ | \ p_1\vec{x_1}=p_2\vec{x_2}=\cdots=p_n\vec{x_n}\rangle.
\]
The \emph{canonical element} $\vec{c}=p_1\vec{x_1}=\cdots =p_n \vec{x_n}$, and the \emph{dualizing element} $\vec{\omega}=(n-2)\vec{c}-\sum_{i=1}^n\vec{x_i}$. We have the \emph{degree map}
\[
\delta: L(\bm p)\twoheadrightarrow \mathbb Z, \ \vec{x_i}\mapsto p/p_i,
\]
where $p=\mathrm{l.c.m.}(p_1,\cdots,p_n)$.

Let $k$ be an algebraically closed field, and let $\bm \lambda=(\lambda_1=\infty, \lambda_2=0, \lambda_3=1, \lambda_4, \cdots, \lambda_n)$ be an $n$-tuple of different points in $\mathbb P^{1}_k$. One define an $L(\bm p)$-graded ring
\[
S=S(\bm p, \bm{\lambda})=k[X_1,X_2,\cdots,X_n]/(X_i^{p_i}-X_2^{p_2}+\lambda_i X_1^{p_1})_{3 \leq i \leq n}
\]
graded with $\mathrm{deg}(X_i)=\vec{x_i}$. Let
\[
   \mathbb X=\mathbb X(\bm p, \bm \lambda)=\{\mathfrak{p} \text{ homogeneous prime ideal } | \ \mathfrak{p} \nsubseteq (x_1, \cdots x_n)\}
\]
be a topological space equipped with the \emph{Zariski topology}, which has a topological basis 
\[
   \{D(f):=\{\mathfrak{p} \in \mathbb X \ | \ f \notin \mathfrak{p}\} \ | \ f \in (x_1, \cdots , x_n) \text{ homogeneous}\}.
\]
One define the \emph{structure sheaf} $\mathcal O_{\mathbb X}$ on $\mathbb X$ by setting $\mathcal O_{\mathbb X}(D(f))=S_f$. $\mathcal O_{\mathbb X}$ is a sheaf of $L(\bm p)$-graded rings. The \emph{weighted projective line} is by definition the $L(\bm p)$-graded locally ringed space $(\mathbb X, \mathcal O_{\mathbb X})$. The points $\lambda_1, \cdots, \lambda_n$ are called the \emph{exceptional points}, and the other points on $\mathbb X$ are called the \emph{ordinary points}. 

Note that $L(\bm{p})$ acts on the the category of $L(\bm p)$-graded $S$-modules by \emph{grading shift}: for an $L(\bm p)$-graded module $M$, and $\vec l \in L(\bm p)$ we denote by $M(\vec l)$ the $S$-module $M$ with the new grading $M(\vec l)_{\vec l'}=M_{\vec l + \vec l'}$. It induces the grading shift functor on the category of \emph{$L(\bm p)$-graded $\mathcal O_{\mathbb X}$-module}. Precisely, for an $\mathcal O_{\mathbb X}$-module $F$, and $\vec l \in L(p)$, one define $F(\vec l)$ by $(F(\vec l))(U)=(F(U))(\vec l)$ for any open subset $U$ of $\mathbb X$.

A \emph{coherent sheaf} on $\mathbb X$ is by definition an $L(\bm p)$-graded $\mathcal O_{\mathbb X}$-module $F$, where for each $x \in \mathbb X$ there is an open neighbourhood $U$ of $x$ and an exact sequence
\[
   \bigoplus_{j=1}^{m'} \mathcal O_{\mathbb X}(\vec l_j)|_U \to \bigoplus_{i=1}^m \mathcal O_{\mathbb X}(\vec l_i)|_U \to F|_U \to 0.
\]

By \cite{geigle1987class}, the category $\mathrm{coh} \text{-} \mathbb X$ is an Ext-finite hereditary abelian category with Serre duality
\[
\mathrm{Ext}_{\mathbb X}^1(F,G)\xrightarrow{\sim} D \mathrm{Hom}_{\mathbb X}(G,F(\vec \omega)),
\]
where $D$ denotes the dual functor $\mathrm{Hom}_k (-,k)$ of $k$-vector spaces.

A coherent sheaf $F$ is called \emph{locally free} (or a \emph{vector bundle}), if for any $x\in \mathbb X$, there exists an open neighbourhood $U\ni x$, $r\in \mathbb N$, and $\vec l_1,\cdots ,\vec l_r \in L(\bm p)$, such that $F|_{U}\cong \bigoplus_{i=1}^r\mathcal O_{\mathbb X}(\vec l_i)|_{U}$. We call $r$ the \emph{rank} of $F$, denoted by $\mathrm{rk}(F)$. We write $\mathrm{vect}(\mathbb X)$ for the subcategory of vector bundles. Rank one vector bundles are called \emph{line bundles}. Each line bundle has the form $\mathcal O_{\mathbb{X}}(\vec l)$, where $\vec l \in L(\bm p)$. Each rank $n$ vector bundle $F$ has a filtration
\[
0=F_0 \subset F_1 \subset \cdots \subset F_n=F,
\]
whose factors $F_i/F_{i-1}$ are line bundles.

Each $F \in \mathrm{coh} \text{-} \mathbb X$ has a unique maximal finite length subsheaf, denoted by $\mathrm{t}F$. For any $F \in \mathrm{coh} \text{-} \mathbb X$, the short exact sequence
\[
0 \to \mathrm{t}F \to F \to F/\mathrm{t}F \to 0
\]
splits, and $F/\mathrm{t}F$ is a vector bundle.

We denote by $\mathrm{coh}_0 \text{-} \mathbb X$ the subcategory consisting of all \emph{finite length sheaves} (i.e. \emph{torsion sheaves}). A coherent sheaf $F$ is torsion if and only if its \emph{support} 
\[
   \mathrm{supp}(F):=\{x \in \mathbb X \ | \ F_x \neq 0 \}
\]
is a finite set. And in this case, we have $F=\oplus_{x \in \mathrm{supp}(F)} F_x$.

All the simple sheaves on $\mathbb X$ are isomorphic to those $S_\lambda, S_{i,j}$ defined as follows:
\[
 0 \to \mathcal O_{\mathbb X}\xrightarrow{u_\lambda} \mathcal O_{\mathbb X}(\vec c)\to S_\lambda \to 0,
\]
\[
 0 \to \mathcal O_{\mathbb X}((j-1)\vec {x_i})\xrightarrow{x_i} \mathcal O_{\mathbb X}(j\vec{x_i})\to S_{i,j} \to 0,
\]
where $u_\lambda=x_2^{p_2}-\lambda x_1^{p_1}$, $\lambda \notin \bm \lambda$, $j=1,2,\cdots,p_i \in \mathbb Z/ \mathbb Z_{p_i}$. \\
Note that for any $\vec l =\sum_{k=1}^n l_k \vec{x_k}+l \vec c \in L(\bm p)$ written in normal form, we have $S_\lambda (\vec l) \cong S_\lambda$ and $S_{i,j}(\vec l) \cong S_{i,j+l_k}$. 

The subcategory $\mathrm{tor}_{\lambda} \text{-} \mathbb X:=\{ F \in \mathrm{coh}_0 \text{-} \mathbb X \ | \ \mathrm{supp}(F)=\{ \lambda \} \}$ is a connected uniserial category; see \cite[Proposition~2.5]{geigle1987class}. Obviously, it is a wide subcategory.

\begin{lem} \label{bundle_tor}
   Let $E$ be a rank $r$ vector bundle, and let $m \in \mathbb N_+$. Then the following statements hold.
   
   (1) We have the following exact sequence
   \[
   0 \to E(-m \vec x_i) \xrightarrow{x_i^{m}} E \to \oplus_{j=1}^{r}U_j \to 0,
   \]
   where each $U_j$ is an indecomposable torsion sheaf concentrated on $\lambda_i$ with length $m$. 
   
   (2) Let $T \in \mathrm{tor}_{\lambda_i} \text{-} \mathbb X$ with $l(T)\leq m$. Then we have 
   \[
      \mathrm{Hom}_{\mathbb X}(E,T) \cong \oplus_{j=1}^{r}\mathrm{Hom}_{\mathbb X}(U_j, T),
   \] 
   where $U_1,U_2,\cdots, U_r$ are defined in (1).
\end{lem}
   
\begin{proof}
   (1) For any $\lambda \neq \lambda_i$, the element $x_i$ is invertible in $S(\bm p, \bm{\lambda})_{\lambda}$. So the multiplication $(x_i^m)_{\lambda}:E_{\lambda}(-m \vec x_i) \to E_{\lambda}$ is an isomorphism. 
   
   We write the stalk $E_{\lambda_i} \cong \oplus_{j=1}^r S_{x_i}(\vec l_j)$ as a free graded $S_{x_i}$-module. Then the cokernel of $(x_i^m)_{\lambda_i}:E_{\lambda_i}(-m \vec x_i) \to E_{\lambda_i}$ is isomorphic to 
   \[
      \bigoplus_{j=1}^r S_{x_i}(\vec l_j)/ x_i^m S_{x_i}(\vec l_j-m \vec x_i).
   \] 
   Let $U_j$ be the sheafification of $S_{x_i}(\vec l_j)/ x_i^m S_{x_i}(\vec l_j-m \vec x_i)$, then $U_j$ is an indecomposable torsion sheaf concentrated on $\lambda_i$ with length $m$. And we have the exact sequence 
   \[
      0 \to E(-m \vec x_i) \xrightarrow{x_i^{m}} E \to \oplus_{j=1}^{r}U_j \to 0.
   \]
   
   (2) For a vector bundle $E$, the stalk $E_{\lambda_i}$ is a free graded $S_{(x_i)}$-module. And $T_{\lambda_i}$ can be annihilated by the ideal $x_i^m S_{(x_i)}$. Then by the universal property of direct limit, we have
   \begin{align*}
   \mathrm{Hom}_{\mathbb X}(E, T)& \cong \mathrm{Hom}_{S_{(x_i)}}(E_{\lambda_i}, T_{\lambda_i}) \\
   & \cong \mathrm{Hom}_{S_{(x_i)}}(E_{\lambda_i}/x_i^{m}E_{\lambda_i}(- m \vec x_i), T_{\lambda_i}) \\
   & \cong \mathrm{Hom}_{\mathbb X}(\oplus_{j=1}^{r}U_j, T).
   \end{align*}
   The proof is complete.
\end{proof}

Weighted projective lines are classified by the value of $\delta(\vec \omega)$:

(1) If $\delta(\vec \omega)<0$, then $\mathbb X$ is called of \emph{domestic} type. Precisely, $\bm{p}=(p,q)$, $(2,2,n)$, $(2,3,3)$, $(2,3,4)$, $(2,3,5)$.

(2) If $\delta(\vec \omega)=0$, then $\mathbb X$ is called of \emph{tubular} type. Precisely, $\bm{p}=(2,2,2,2)$, $(3,3,3)$, $(2,4,4)$, $(2,3,6)$.

(3) If $\delta(\vec \omega)>0$, then $\mathbb X$ is called of \emph{wild} type.

It is well known that a weighted projective line of domestic type is derived equivalent to a tame hereditary algebra; see the following theorem:

\begin{thm}  [{\cite[Theorem~3.5]{lenzing2011weighted}}]\label{domestic}
Let $\mathbb X$ be a weighted projective line of domestic type, then the following properties hold: 

(1) Each indecomposable bundle is exceptional. 

(2) Let $T$ be the direct sum of all indecomposable bundles $E$ satisfying $0 \leq \mu(E) < -\delta(\omega)$. Then $T$ is a tilting object.

(3) The endomorphism ring of $T$ is isomorphic to the path algebra of an extended Dynkin quiver, which is either of type $\tilde{\mathbb A}_{p,q}$ or of type $\tilde{\mathbb D}_n$ ($n \geq 4$), $\tilde{\mathbb E}_6$, $\tilde{\mathbb E}_7$ or $\tilde{\mathbb E}_8$ with bipartite orientation.
\end{thm}

\section{The $\vec c$ - invariant wide subcategories of $\mathrm{coh} \text{-} \mathbb X$}

Let $\mathbb X$ be a weighted projective line of any type, we consider wide subcategories of $\mathrm{coh} \text{-} \mathbb X$ containing both a vector bundle and a non-exceptional indecomposable torsion sheaf. They are exactly socalled $\vec c$-invariant wide subcategories which containing a vector bundle. And each of them is the perpendicular category of a torsion exceptional sequence; see Theorem~\ref{c_inv}.

\begin{defn}
A wide subcategory $\mathcal W$ of $\mathrm{coh} \text{-} \mathbb X$ is called $\vec c$-invariant, if $\mathcal W = \mathcal W(\vec c)$, where $\mathcal W(\vec c):= \{F(\vec c) \ | \ F \in \mathcal W  \}$. We denote by $\mathrm{Wid}(\mathrm{coh} \text{-} \mathbb X)^{\vec c}$ the set of $\vec c$-invariant wide subcategories of $\mathrm{coh} \text{-} \mathbb X$.
\end{defn}

Note that for a torsion sheaf $T$, we have $T \cong T(\vec c)$. So a wide subcategory formed by torsion sheaves is always $\vec c$-invariant. The next proposition states that a non-exceptional indecomposable torsion sheaf will force a wide subcategory to be $\vec c$-invariant.

\begin{prop} \label{contain_non_exc}
Let $\mathcal W$ be a wide subcategory of $\mathrm{coh} \text{-} \mathbb X$ containing a vector bundle. Then $\mathcal W$ contain an indecomposable non-exceptional torsion sheaf if and only if $\mathcal W$ is $\vec c$-invariant.
\end{prop}

\begin{proof}
   $\Rightarrow$: Take a vector bundle $E$ in $\mathcal W$, by Lemma~\ref{bundle_tor}(1), we have exact sequences
   \begin{subequations} \label{ecu}
      \begin{gather}
         0 \to E \xrightarrow{x_i^{p_i}} E(\vec c) \to \oplus_{j=1}^{\mathrm{rk}(E)}U_j \to 0  \label{exact_weighted} \\
         0 \to E \xrightarrow{x_2^{p_2}- \lambda x_1^{p_1}} E(\vec c) \to \mathrm{rk}(E)S_{\lambda} \to 0 \label{exact_ordinary}
         \end{gather}
   \end{subequations}
   where $U_1,U_2,\cdots$ are indecomposable sheaves concentrated on $\lambda_i$ of length $p_i$, and $S_{\lambda}$ is the simple sheaf concentrated on an ordinary point $\lambda$. 
   
   Let $U \in \mathcal W$ be an indecomposable non-exceptional torsion sheaf. If $U$ is concentrated on the point $\lambda_i$, by Lemma~\ref{uni_exc} and Lemma~\ref{wide_length}, we can assume that $l(U)=p_i$. By Lemma~\ref{uni_hom}, we have $\mathrm{Hom}_{\mathbb X}(U_j,U)\neq 0$ for any $j=1,2,\cdots, \mathrm{rk}(E)$. Take a nonzero morphism 
   $f: U_j \to U$. Since $\mathcal W$ is wide, we have $\mathrm{Im}f \in \mathcal W$ and $\mathrm{Coker}f \in \mathcal W$. By the short exact sequence 
   \[
      0 \to \mathrm{Coker}f \to U_j \to \mathrm{Im}f \to 0,
   \]
   we have $U_j \in \mathcal W$ for any $j=1,2,\cdots, \mathrm{rk}(E)$. From exact sequence~(\ref{exact_weighted}), we obtain that $E(\vec c) \in \mathcal W$. 
   
   If $U$ is concentrated on an ordinary point $\lambda$, by Lemma~\ref{wide_length}, we can assume that $U \cong S_{\lambda}$. From the exact sequence (\ref{exact_ordinary}), we obtain $E(\vec c) \in \mathcal W$. 

   $\Leftarrow$: It is obvious by the exact squences (\ref{ecu}) and Lemma~\ref{uni_exc}.
\end{proof}

We have the classification theorem of $\vec c$-invariant wide subcategories: they are completely determined by their ``stalks''.

\begin{thm} \label{c_inv}
Let $\mathbb X$ be a weighted projective line with $n$ exceptional points, and let $\mathcal W$ be a $\vec c$ -invariant wide subcategory of $\mathrm{coh} \text{-} \mathbb X$ containing a vector bundle. Then

(1) For any $i=1,2,\cdots, n$, there exists an exceptional sequence $E^i=(E_1^i,\cdots,E_{r_i}^i)$ in $\mathrm{tor}_{\lambda_i} \text{-} \mathbb X$, such that
\[
   \mathcal W \cap \mathrm{tor}_{\lambda_i} \text{-} \mathbb X =(E^i)^{\perp} \cap \mathrm{tor}_{\lambda_i} \text{-} \mathbb X.
\]

(2) $\mathcal W=(E^1,\cdots,E^n)^{\perp}$, where $E^1, \cdots , E^n$ are defined in (1).
\end{thm}

\begin{proof}
(1) By the exact sequence (\ref{exact_weighted}) and Proposition~\ref{uniserial}(3), we have 
\[
   \mathcal W \cap \mathrm{tor}_{\lambda_i} \text{-} \mathbb X \in \mathrm{Nexc}(\mathrm{tor}_{\lambda_i} \text{-} \mathbb X).
\]
And then it follows from Proposition~\ref{thick_uniserial}.

(2) First, we claim that $\mathcal W \subseteq (E^1,\cdots,E^n)^{\perp}$. By Lemma~\ref{bundle_tor}(1), for any indecomposable bundle $F \in \mathcal W$, we have a short exact sequence
\[
0 \to F \xrightarrow{x_i^{p_i}} F(\vec{c}) \to \oplus_{j=1}^{\mathrm{rk}(E)}U_j \to 0,
\]
where $U_1,U_2, \cdots$ are indecomposable sheaves concentrated on $\lambda_i$ of length $p_i$. Since $\mathcal W$ is $\vec c$-invariant and wide, we have that $U_j \in \mathcal W$. So $U_j \in \mathcal W \cap \mathrm{tor}_{\lambda_i} \text{-} \mathbb X \subseteq (E^i)^\perp$. By Serre duality, 
\[
   \mathrm{Hom}_{\mathbb X}(U_j, E_k^i(\vec{\omega}))=D\mathrm{Ext}_{\mathbb X}^1(E_k^i, U_j)=0
\]
for $k=1, \cdots, r_i$. By Lemma~\ref{uni_exc}, we have $l(E_k^i)<p_i$ since $E_k^i$ is exceptional. And then by Lemma~\ref{bundle_tor}(2), we have 
\[
   \mathrm{Hom}_{\mathbb X}(F, E_k^i(\vec{\omega}))=\mathrm{Hom}_{\mathbb X}(\oplus_{j=1}^{\mathrm{rk}(E)}U_j, E_k^i(\vec{\omega}))=0,
\]
which implies that $\mathrm{Ext}_{\mathbb X}^1(E_k^i, F)=0$. So $F \subseteq (E^i)^{\perp}$, and then $\mathcal W \subseteq (E^1,\cdots,E^n)^{\perp}$. 

For the other direction $\mathcal W \supseteq (E^1,\cdots,E^n)^{\perp}$, we only need to prove that $\prescript{\perp}{}{\mathcal W} \subseteq \langle E^1,\cdots,E^n \rangle $. By Proposition~\ref{perp_cal}(2), $(\langle E^i \rangle, (E^i)^{\perp} \cap \mathrm{tor}_{\lambda_i} \text{-} \mathbb X)$ is an extension orthogonal pair for $\mathrm{tor}_{\lambda_i} \text{-} \mathbb X$, so we have
\begin{align*}
\prescript{\perp}{}{\mathcal W} \cap \mathrm{tor}_{\lambda_i} \text{-} \mathbb X & \subseteq \prescript{\perp}{}{(\mathcal W \cap \mathrm{tor}_{\lambda_i} \text{-} \mathbb X)} \cap \mathrm{tor}_{\lambda_i} \text{-} \mathbb X \\
&=\prescript{\perp}{}{((E^i)^{\perp} \cap \mathrm{tor}_{\lambda_i} \text{-} \mathbb X)} \cap \mathrm{tor}_{\lambda_i} \text{-} \mathbb X \\
&=\langle E^i \rangle.
\end{align*}
By Lemma~\ref{bundle_tor}(1), for any ordinary point $\lambda$, the simple sheaf $S_\lambda$ is in $\mathcal W$. For any vector bundle $B$ and any torsion sheaf $U_\lambda$ concentrated on $\lambda$, we have $\mathrm{Hom}_{\mathbb X}(B,S_\lambda) \neq 0$ and $\mathrm{Hom}_{\mathbb X}(U_\lambda,S_\lambda) \neq 0$. So $B \notin \prescript{\perp}{}{\mathcal W}$, $U_\lambda \notin \prescript{\perp}{}{\mathcal W}$. Then any indecomposable sheaf in $\prescript{\perp}{}{\mathcal W}$ is a torsion sheaf concentrated on an exceptional point. Finally, we obtain
\[
   \prescript{\perp}{}{\mathcal W}=\coprod_{i=1}^n \prescript{\perp}{}{\mathcal W} \cap \mathrm{tor}_{\lambda_i} \text{-} \mathbb X \subseteq \coprod_{i=1}^n \langle E^i \rangle =\langle E^1, \cdots , E^n \rangle.
\]
The proof is complete.
\end{proof}

\begin{cor} \label{cor_c_inv}
   The set of all $\vec c$-invariant wide subcategories of $\mathrm{coh} \text{-} \mathbb X$ can be written as the following disjoint union of two sets:
   \[
      \mathrm{Wid}(\mathrm{coh} \text{-} \mathbb X)^{\vec c}= \mathrm{Wid}(\mathrm{coh}_0 \text{-} \mathbb X) \sqcup \{ \mathcal W^{\perp} | \mathcal W \in \mathrm{Exc}(\mathrm{coh}_0 \text{-} \mathbb X) \}.
   \]
\end{cor}

\begin{proof}
   By Proposition~\ref{perp_cal}(2), for any $\mathcal W \in \mathrm{Exc}(\mathrm{coh}_0 \text{-} \mathbb X)$, we have $(\mathcal W, \mathcal W^{\perp})$ is a complete Ext-orthogonal pair in $\mathrm{coh} \text{-} \mathbb X$. So there must exist a vector bundle in $\mathcal W^\perp$, hence the union is disjoint.
   
   Obviously $\mathrm{Wid}(\mathrm{coh}_0 \text{-} \mathbb X) \subseteq \mathrm{Wid}(\mathrm{coh} \text{-} \mathbb X)^{\vec c}$. Since the perpendicular category of a $\vec c$-invariant subcategory is $\vec c$-invariant, we have 
   \[
      \{ \mathcal W^{\perp} | \mathcal W \in \mathrm{Exc}(\mathrm{coh}_0 \text{-} \mathbb X) \} \subseteq \mathrm{Wid}(\mathrm{coh} \text{-} \mathbb X)^{\vec c}.
   \]
   
   For any $\mathcal W \in \mathrm{Wid}(\mathrm{coh} \text{-} \mathbb X)^{\vec c} \setminus  \mathrm{Wid}(\mathrm{coh}_0 \text{-} \mathbb X)$, by Theorem~\ref{c_inv}(2), we have that $\mathcal W=(E^1,\cdots,E^n)^{\perp}$, where each $E^i$ is an exceptional sequence in $\mathrm{tor}_{\lambda_i} \text{-} \mathbb X$. So $\mathrm{Wid}(\mathrm{coh} \text{-} \mathbb X)^{\vec c} \subseteq \mathrm{Wid}(\mathrm{coh}_0 \text{-} \mathbb X) \cup \{ \mathcal W^{\perp} | \mathcal W \in \mathrm{Exc}(\mathrm{coh}_0 \text{-} \mathbb X) \}$.
\end{proof}

Further, the following result about the structure of $\vec c$-invariant wide subcategory is known for experts, after \cite[Theorem 9.5]{Werner1991Perpendicular}. Here, we provide a proof.

\begin{prop}
Let $\mathbb X$ be a weighted projective line of type $(p_1,\cdots,p_n)$, and let $E$ be an exceptional torsion sheaf in $\mathrm{coh} \text{-} \mathbb X$ concentrated on $\lambda_i$, then
\[
E^{\perp} \simeq \mathrm{coh} \text{-}\mathbb X' \coprod k \vec{\mathbb A}_{l(E)-1}\mathrm{-}\mathrm{mod},
\]
where $\mathbb X'$ is a weighted projective line of type $(p_1,\cdots,p_i-l(E),\cdots,p_n)$.
\end{prop}

\begin{proof}
   Assume $E \cong S_{i,j}^{[m]}$, then we have $m<p_i$ by Lemma~\ref{uni_exc}. Denote $\mathcal A:=(S_{i,j}, \tau S_{i,j}, \cdots, \tau^{m-1} S_{i,j})^\perp$ and denote $\mathcal B:= \langle \tau S_{i,j}, \cdots, \tau^{m-1} S_{i,j} \rangle$. We claim that $E^\perp =\mathcal A \coprod \mathcal B$. 

   Obviously $E \in \langle S_{i,j}, \tau S_{i,j}, \cdots, \tau^{m-1} S_{i,j} \rangle$, so $\mathcal A \subseteq E^\perp$. It is easy to check that $\tau S_{i,j}, \cdots, \tau^{m-1} S_{i,j} \in E^\perp$, so $\mathcal B \subseteq E^\perp$.

   For any indecomposable sheaf $F$, we will prove that $F \in \mathcal A$ or $F \in \mathcal B$. If $F \in \mathrm{vect}(\mathbb X)$, since $\mathrm{Hom}_{\mathbb X}(T,F)=0$ for any $T \in \mathrm{coh}_0 \text{-} \mathbb X$, by the homological long exact sequence, we have that
   \[
      \mathrm{Ext}_{\mathbb X}^1(-,F):\mathrm{tor}_{\lambda_i} \text{-} \mathbb X \to k \text{-mod}
   \]
   is a contravariant exact functor. So $\mathrm{Ext}_{\mathbb X}^1(E,F)=0$ implies that for any $S \in \mathrm{CF}(E)$, we have $\mathrm{Ext}_{\mathbb X}^1(S,F)=0$. That is $\mathrm{Ext}_{\mathbb X}^1(\tau^{l} S_{i,j},F)=0$ for $l=0,1,\cdots, m-1$. So $F \in \mathcal A$. If $F \in \mathrm{tor}_\lambda (\mathbb X)$ with $\lambda \notin \{\lambda_1, \cdots, \lambda_n \}$, obviously $F \in \mathcal A$. If $F \in \mathrm{tor}_{\lambda_i} (\mathbb X)$, then by Proposition~\ref{uni_exc_perp}, we have either $F \in \mathcal A$ or $F \in \mathcal B$. Now we have proved that $F \in \mathcal A$ or $F \in \mathcal B$.

   By Proposition~\ref{uni_exc_perp}, for any $A' \in \mathcal A \cap \mathrm{tor}_{\lambda_i} \text{-} \mathbb X$, $B \in \mathcal B$, we have 
   \[
      \mathrm{Hom}_{\mathbb X}(A',B)=0=\mathrm{Hom}_{\mathbb X}(B,A').
   \]
   And then by Lemma~\ref{bundle_tor}(2), for any $A \in \mathcal A$, $B \in \mathcal B$, we have $\mathrm{Hom}_{\mathbb X}(A,B)=0$. Obviously, $\mathrm{Hom}_{\mathbb X}(B,A)=0$. So $\mathrm{Hom}_{\mathbb X}(\mathcal A, \mathcal B)=\mathrm{Hom}_{\mathbb X}(\mathcal B, \mathcal A)=0$. Now our claim $E^\perp =\mathcal A \coprod \mathcal B$ is proved. 
      
   It remains to prove that $\mathcal A \simeq \mathrm{coh} \text{-}\mathbb X'$ and $\mathcal B \simeq k \vec{\mathbb A}_{m-1}$. The former can be concluded by \cite[Theorem~9.5]{Werner1991Perpendicular}. For the latter, we only need to prove that $P=\oplus_{t=1}^{m-1} \tau^t S_{i,j}^{[m-t]}$ is a projective generator of $\mathcal B$. By Serre duality, 
   \[
      \mathrm{Ext}_{\mathbb X}^1(\tau^t S_{i,j}^{[m-t]}, M) \cong D \mathrm{Hom}_{\mathbb X}(M, \tau^{t+1} S_{i,j}^{[m-t]}) =0,
   \]
   where $t=1,2, \cdots , m-1$, $M \in \mathcal B$. So $P$ is projective in $\mathcal B$. By Proposition~\ref{proj_gen}, $P$ is a generator of $\mathcal B$. The proof is complete.
\end{proof}

\section{Wide subcategories of a domestic weighted projective line}

In this section, we classify all wide subcategories for category of coherent sheaves over a domestic weighted projective line. And study the poset of those wide subcategories.

In \cite{dichev2009thick}, wide subcategories of a tame hereditary algebra are classified well: any wide subcategory is either generated by an exceptional sequence, or consists of only regular modules. The following proposition is the version of weighted projective lines for the classification theorem above.

\begin{prop} \label{thick_domestic_bundle}
Let $\mathbb X$ be a weighted projective line of domestic type and $\mathcal W$ a wide subcategory of $\mathrm{coh} \text{-} \mathbb X$. If $\mathcal W$ contains a vector bundle, then $\mathcal W \in \mathrm{Exc}(\mathrm{coh} \text{-} \mathbb X)$.
\end{prop}
\begin{proof}
Let $F \in \mathcal W$ be a vector bundle, and take $E \in F^{\perp_{\mathcal W}}$ be indecomposable. If $E \in \mathrm{vect}(\mathbb X)$, by Proposition~\ref{domestic}(1), we know that  $E$ is exceptional. If $E \in \mathrm{tor}(\mathbb X)$, suppose that $\mathrm{supp}(E)=\{\lambda \}$, we will show $E$ is exceptional. By the line bundle filtration of vector bundles, $F$ has a quotient line bundle $L$. If $\lambda \notin \{\lambda_1, \cdots, \lambda_n \}$, then $\mathrm{Hom}_{\mathbb X}(L,E) \neq 0$, hence $\mathrm{Hom}_{\mathbb X}(F,E) \neq 0$, a contradiction. If $\lambda=\lambda_i$ for some $\lambda_i$ and $E$ is not exceptional, then Lemma~\ref{uni_exc} implies that $\mathrm{CF}(E)=\{S_{i,1}, \cdots, S_{i,p_i} \}$. So $\mathrm{Hom}_{\mathbb X}(L,E) \neq 0$, and then $\mathrm{Hom}_{\mathbb X}(F,E) \neq 0$, a contradiction.

So each indecomposable sheaf $E \in F^{\perp_{\mathcal W}} \neq 0$ is exceptional. By Lemma~\ref{wide_gen_exc}, there exists an exceptional sequence $(E_1, E_2, \cdots , E_r)$ in $F^{\perp_{\mathcal W}}$, such that 
\[
   F^{\perp_{\mathcal W}}=\langle E_1, E_2, \cdots , E_r \rangle.
\] 
Hence $\mathcal W= \langle F, E_1, E_2, \cdots , E_r \rangle$. Obviously $(F, E_1, E_2, \cdots , E_r)$ is an exceptional sequence in $\mathcal W$, so $\mathcal W \in \mathrm{Exc}(\mathrm{coh} \text{-} \mathbb X)$.
\end{proof}

\begin{rmk}
It is known that for a weighted projective line $\mathbb X$ of any type, the perpendicular category of a vector bundle is equivalent to the module category of a hereditary algebra $A$. Moreover if $\mathbb X$ is domestic, then $A$ is representation-finite. It is also known that any wide subcategory of a representation-finite hereditary algebra can be generated by an exceptional sequence. The last proposition can also be proved in this way.
\end{rmk}

\begin{cor} \label{thick_nexc}
   Let $\mathbb X$ be a weighted projective line of domestic type. Then $\mathrm{NExc}( \mathrm{coh} \text{-} \mathbb X) \subseteq \mathrm{Wid}(\mathrm{coh}_0 \text{-} \mathbb X) \subseteq \mathrm{Wid}(\mathrm{coh} \text{-} \mathbb X)^{\vec c}$.
\end{cor}

Next we consider the poset structure of $\mathrm{Wid}(\mathrm{coh} \text{-} \mathbb X)$, which is ordered by inclusion. We first recall some basic notions for generalized noncrossing partitions. Let $\mathcal H$ be module category of a hereditary algebra or coherent sheaf category of a weighted projective line. There is a partial order called \emph{absolute order} on the Weyl group $W$ of $\mathcal H$:
\[
u<v \ \ \mathrm{provided} \ \ l(u)+l(u^{-1}v)=l(v),
\]
where $l(w)$ is the \emph{absolute length} of $w \in W$, defined to be the minimal integer $r \geq 0$ such that $w$ can be written as a product $w=t_1 \cdots t_r$ of \emph{reflections} $t_i$. The Serre duality functor of $\mathcal H$ induces the Coxeter element $c_{\mathcal H} \in W$ satisfying
\[
\langle x,c_{\mathcal H}(y) \rangle =- \langle y,x \rangle, \ \forall x,y \in K_0(\mathcal H).
\]
For an element $c \in W$, one define the poset of \emph{noncrossing partitions}
\[
\mathrm{NC}(W,c):=\{w \in W | 1 \leq w \leq c \}.
\]
Then we have an order preserving map $\mathrm{cox}: \mathrm{Exc}(\mathcal H) \to \mathrm{NC}(W,c_{\mathcal H})$, defined by $\mathrm{cox}(\langle E_1,\cdots,E_r \rangle)=s_{E_1}\cdots s_{E_r}$, where $(E_1,\cdots,E_n)$ is an exceptional sequence. From \cite[Theorem]{crawley1993exceptional} and \cite[Theorem 4.3.1]{mccullough2004exceptional} we know that the map ``cox'' is well-defined. For a ditailed introduction to root system of hereditary algebras and weighted projective lines, we refer to \cite{hubery2016categorification} and \cite[Section 2]{2016On}, respectively.

The following theorem states that the map ``cox'' is a bijection for hereditary algebras. The case of representation finite type and tame hereditary algebras were first studied in \cite{ingalls2009noncrossing}.

\begin{thm} [{\cite[Theorem~1.2]{hubery2016categorification}}] \label{hubery_krause}

Let $A$ be a finite dimensional hereditary algebra with Weyl group $W$ and $c_A$ the Coxeter element of $W$. Then
\[
\mathrm{cox}: \mathrm{Exc}(A \text{-} \mathrm{mod}) \to \mathrm{NC}(W,c_A)
\]
is an isomorphism of posets.
\end{thm}

In fact, by Proposition \ref{bruning}, we have the following commutative diagram.
\[
   \xymatrix{\mathrm{Exc}(A \text{-} \mathrm{mod}) \ar[rd]^{\mathrm{cox}} \ar[r]^{\sim} & \mathrm{Exc}(D^b(A \text{-} \mathrm{mod})) \ar[d]^{\mathrm{cox}}\\
    & W }
\]
In the commutative diagram above, we denote by $\mathrm{Exc}(D^b(A \text{-} \mathrm{mod}))$ the set of thick subcategories of $D^b(A \text{-} \mathrm{mod})$ which can be generated by an exceptional sequence. And we still use the notation ``cox'' to present the map 
\begin{align*}
   \mathrm{Exc}(D^b(A \text{-} \mathrm{mod})) &\to W \\
   \langle E_1, \cdots, E_r \rangle &\mapsto s_{E_1} \cdots s_{E_r}.
\end{align*}
Since a domestic weighted projective line $\mathbb X$ is derived equivalent to a tame hereditary algebra (Theorem \ref{domestic}), the last theorem also holds true for $\mathrm{coh} \text{-} \mathbb X$.

\begin{prop} \label{domestic_order}
Let $\mathbb X$ be a weighted projective line of domestic type with Weyl group $W$ and $c_{\mathbb X}$ the Coxeter element of $\mathrm{coh} \text{-} \mathbb X$. Then
\[
\mathrm{cox}: \mathrm{Exc}(\mathrm{coh} \text{-} \mathbb X) \to \mathrm{NC}(W,c_{\mathbb X})
\]
is an isomorphism of posets.
\end{prop}

\begin{rmk}
Proposition~\ref{domestic_order} also holds true for weighted projective lines of wild type; see \cite[Theorem 7.2.5]{yahiatene2020hurwitz}. But it is still unknown for those of tubular type.
\end{rmk}

For a poset $P$, we denote by $\mathrm{Ord}(P)=\{(p_1,p_2) \in P \times P \ | \ p_1 \leq p_2 \}$. Then we have the following statement of our main theorem:

\begin{thm} \label{main}
   Let $\mathbb X$ be a weighted projective line of domestic type. Then we have
   \[
      \mathrm{Wid}(\mathrm{coh} \text{-} \mathbb X)=\mathrm{Exc}(\mathrm{coh} \text{-} \mathbb X) \cup \mathrm{Wid}(\mathrm{coh} \text{-} \mathbb X)^{\vec c},
   \]
   and
   \[
      \mathrm{Ord}(\mathrm{Wid}(\mathrm{coh} \text{-} \mathbb X))=\mathrm{Ord}(\mathrm{Exc}(\mathrm{coh} \text{-} \mathbb X)) \cup \mathrm{Ord}(\mathrm{Wid}(\mathrm{coh} \text{-} \mathbb X)^{\vec c}).
   \]
   In particular, we have a pushout of posets
   \begin{equation} \label{po}
      \begin{split}
         \xymatrix{
   \mathrm{Exc}(\mathrm{coh} \text{-} \mathbb X) \cap \mathrm{Wid}(\mathrm{coh} \text{-} \mathbb X)^{\vec c} \ar@{^(->}[r] \ar@{^(->}[d] & \mathrm{Wid}(\mathrm{coh} \text{-} \mathbb X)^{\vec c} \ar@{^(->}[d]\\
   \mathrm{Exc}(\mathrm{coh} \text{-} \mathbb X) \ar@{^(->}[r] & \mathrm{Wid}(\mathrm{coh} \text{-} \mathbb X).
   }
      \end{split}
   \end{equation}
\end{thm}
\begin{proof}
   By Corollary \ref{thick_nexc}, we have $\mathrm{Wid}(\mathrm{coh} \text{-} \mathbb X)=\mathrm{Exc}(\mathrm{coh} \text{-} \mathbb X) \cup \mathrm{Wid}(\mathrm{coh} \text{-} \mathbb X)^{\vec c}$. So we only need to show that 
   \[
      \mathrm{Ord}(\mathrm{Wid}(\mathrm{coh} \text{-} \mathbb X))=\mathrm{Ord}(\mathrm{Exc}(\mathrm{coh} \text{-} \mathbb X)) \cup \mathrm{Ord}(\mathrm{Wid}(\mathrm{coh} \text{-} \mathbb X)^{\vec c}).
   \]

   Obviously $\mathrm{Ord}(\mathrm{Wid}(\mathrm{coh} \text{-} \mathbb X))\supseteq \mathrm{Ord}(\mathrm{Exc}(\mathrm{coh} \text{-} \mathbb X)) \cup \mathrm{Ord}(\mathrm{Wid}(\mathrm{coh} \text{-} \mathbb X)^{\vec c})$. Now let $(\mathcal A, \mathcal B) \in \mathrm{Ord}(\mathrm{Wid}(\mathrm{coh} \text{-} \mathbb X))$. If $(\mathcal A, \mathcal B) \notin \mathrm{Ord}(\mathrm{Exc}(\mathrm{coh} \text{-} \mathbb X))$, then it follows immediately that $\mathcal A \in \mathrm{NExc}(\mathrm{coh} \text{-} \mathbb X)$ or $\mathcal B \in \mathrm{NExc}(\mathrm{coh} \text{-} \mathbb X)$. We will show that $(\mathcal A, \mathcal B) \in \mathrm{Ord}(\mathrm{Wid}(\mathrm{coh} \text{-} \mathbb X)^{\vec c})$.
   
   If $\mathcal A \in \mathrm{NExc}(\mathrm{coh} \text{-} \mathbb X)$, then Corollary~\ref{thick_nexc} states that $\mathcal A$ is $\vec c$-invariant. And Proposition~\ref{thick_domestic_bundle} implies that $\mathcal A$ contains a non-exceptional torsion sheaf $N$, and then $N \in \mathcal B$. By Proposition~\ref{contain_non_exc}, we obtain $\mathcal B$ is $\vec c$-invariant. So $(\mathcal A, \mathcal B) \in \mathrm{Ord}(\mathrm{Wid}(\mathrm{coh} \text{-} \mathbb X)^{\vec c})$.

   If $\mathcal B \in \mathrm{NExc}(\mathrm{coh} \text{-} \mathbb X)$, then Corollary~\ref{thick_nexc} states that $\mathcal B \in \mathrm{Wid}(\mathrm{coh}_0 \text{-} \mathbb X)$, and then $\mathcal A \in \mathrm{Wid}(\mathrm{coh}_0 \text{-} \mathbb X)$. So $(\mathcal A, \mathcal B) \in \mathrm{Ord}(\mathrm{Wid}(\mathrm{coh} \text{-} \mathbb X)^{\vec c})$. Now we have proved that $\mathrm{Ord}(\mathrm{Wid}(\mathrm{coh} \text{-} \mathbb X))=\mathrm{Ord}(\mathrm{Exc}(\mathrm{coh} \text{-} \mathbb X)) \cup \mathrm{Ord}(\mathrm{Wid}(\mathrm{coh} \text{-} \mathbb X)^{\vec c})$.

   By the lemma bellow, diagram (\ref{po}) is a pushout.
\end{proof}

\begin{lem}
   Let $A$ be a poset, and let $A_1, A_2$ be full subposets of $A$. If $A=A_1 \cup A_2$ and $\mathrm{Ord}(A)=\mathrm{Ord}(A_1) \cup \mathrm{Ord}(A_2)$, then we have a pushout of posets
   \[
      \xymatrix{
         A_1 \cap A_2 \ar@{^(->}[r] \ar@{^(->}[d] & A_1 \ar@{^(->}[d] \\
         A_2 \ar@{^(->}[r] & A.
      }
   \]
\end{lem}
\begin{proof}
   The universal property of the square is easy to check.
\end{proof}

Since $\vec c$-invariant wide subcategories are determined by their stalks, we can describe the poset structure of $\mathrm{Wid}(\mathrm{coh} \text{-} \mathbb X)^{\vec c}$ for a weighted projective line $\mathbb X$ of any type; see Proposition~\ref{c_inv_order}. It relies on the poset structure of $\mathrm{Wid}(\mathrm{tor}_{x} \text{-} \mathbb X)$; see Proposition~\ref{tor_order}.

\begin{prop} \label{c_inv_order}
   Let $\mathbb X$ be a weighted projective line of any type, and let $\mathcal A$, $\mathcal B$ be $\vec c$-invariant wide subcategories of $\mathrm{coh} \text{-} \mathbb X$. Then the following statements hold.

   (1) If $\mathcal A, \mathcal B \in \mathrm{Wid}(\mathrm{coh}_0 \text{-} \mathbb X)$ or $\mathcal A, \mathcal B \notin \mathrm{Wid}(\mathrm{coh}_0 \text{-} \mathbb X)$, then $\mathcal A \leq \mathcal B$ if and only if $\mathcal A \cap \mathrm{tor}_{x} \text{-} \mathbb X \leq \mathcal B \cap \mathrm{tor}_{x} \text{-} \mathbb X$ for any $x \in \mathbb X$.

   (2) If $\mathcal A \in \mathrm{Wid}(\mathrm{coh}_0 \text{-} \mathbb X)$, $\mathcal B \notin \mathrm{Wid}(\mathrm{coh}_0 \text{-} \mathbb X)$, then $\mathcal A < \mathcal B$ if and only if $\mathcal A \cap \mathrm{tor}_{x} \text{-} \mathbb X \leq \mathcal B \cap \mathrm{tor}_{x} \text{-} \mathbb X$ for any $x \in \mathbb X$.
   
   (3) If $\mathcal A \notin \mathrm{Wid}(\mathrm{coh}_0 \text{-} \mathbb X)$, $\mathcal B \in \mathrm{Wid}(\mathrm{coh}_0 \text{-} \mathbb X)$, then $\mathcal A \nleq \mathcal B$.
\end{prop}
\begin{proof}
   We only need to prove the case when $\mathcal A, \mathcal B \notin \mathrm{Wid}(\mathrm{coh}_0 \text{-} \mathbb X)$. By Theorem~\ref{c_inv}(2), we can assume that $\mathcal A=(E^1,\cdots , E^n)^{\perp}$, $\mathcal B=(F^1,\cdots , F^n)^{\perp}$, where $E^i$ and $F^i$ are exceptional sequences in $\mathrm{tor}_{\lambda_i} \text{-} \mathbb X$ for $i=1, 2, \cdots , n$. So $\mathcal A \leq \mathcal B$ if and only if $\langle F^1, \cdots, F^n \rangle \leq \langle E^1, \cdots, E^n \rangle$, if and only if $\langle F^i\rangle \leq \langle E^i \rangle$ for any $i=1,2, \cdots, n$, if and only if $(E^i)^{\perp} \cap \mathrm{tor}_{\lambda_i} \text{-} \mathbb X \leq (F^i)^{\perp} \cap \mathrm{tor}_{\lambda_i} \text{-} \mathbb X$, i.e. $\mathcal A \cap \mathrm{tor}_{\lambda_i} \text{-} \mathbb X \leq \mathcal B \cap \mathrm{tor}_{\lambda_i} \text{-} \mathbb X$ for any $i=1,2, \cdots, n$. For any ordinary point $x$, we always have $\mathcal A \cap \mathrm{tor}_{x} \text{-} \mathbb X =\mathrm{tor}_{x} \text{-} \mathbb X= \mathcal B \cap \mathrm{tor}_{x} \text{-} \mathbb X$.
\end{proof}

By Proposition~\ref{domestic_order}, we can describe the poset structure of $\mathrm{Wid}(\mathrm{tor}_{\lambda_i} \text{-} \mathbb X)$. And by \cite[Corollary 9.2]{krause2021category}, it is isomorphic to $\mathrm{NC}^B(p_i)$, the lattice of non-crossing partitions of type $B$.

\begin{prop} \label{tor_order}
   Let $\mathcal A, \mathcal B$ be wide subcategories of $\mathrm{tor}_{\lambda_i} \text{-} \mathbb X$. Then the following statements hold.

   (1) If $\mathcal A, \mathcal B \in \mathrm{Exc}(\mathrm{tor}_{\lambda_i} \text{-} \mathbb X)$, then $\mathcal A \leq \mathcal B$ if and only if $\mathrm{cox}(\mathcal A) \leq \mathrm{cox}(\mathcal B)$.

   (2) If $\mathcal A \in \mathrm{Exc}(\mathrm{tor}_{\lambda_i} \text{-} \mathbb X), \mathcal B \in \mathrm{NExc}(\mathrm{tor}_{\lambda_i} \text{-} \mathbb X)$, then $\mathcal A < \mathcal B$ if and only if $\mathrm{cox}(\mathcal A) \cdot \mathrm{cox}(\prescript{\perp}{}{\mathcal B} \cap \mathrm{tor}_{\lambda_i} \text{-} \mathbb X)<c_{\mathbb X}$.

   (3) If $\mathcal A, \mathcal B \in \mathrm{NExc}(\mathrm{tor}_{\lambda_i} \text{-} \mathbb X)$, then $\mathcal A \leq \mathcal B$ if and only if $\mathrm{cox}(\prescript{\perp}{}{\mathcal B} \cap \mathrm{tor}_{\lambda_i} \text{-} \mathbb X) \leq \mathrm{cox}(\prescript{\perp}{}{\mathcal A} \cap \mathrm{tor}_{\lambda_i} \text{-} \mathbb X)$.

   (4) If $\mathcal A \in \mathrm{NExc}(\mathrm{tor}_{\lambda_i} \text{-} \mathbb X), \mathcal B \in \mathrm{Exc}(\mathrm{tor}_{\lambda_i} \text{-} \mathbb X)$, then $\mathcal A \nleq \mathcal B$.
\end{prop}
\begin{proof}
   (1) It follows directly from Proposition~\ref{domestic_order}.

   (2) Note that $\mathcal A < \mathcal B$ if and only if $\mathcal A < (\prescript{\perp}{}{\mathcal B} \cap \mathrm{tor}_{\lambda_i} \text{-} \mathbb X)^{\perp}$. By Proposition~\ref{thick_uniserial}, we have $\prescript{\perp}{}{\mathcal B} \cap \mathrm{tor}_{\lambda_i} \text{-} \mathbb X \in \mathrm{Exc}(\mathrm{tor}_{\lambda_i} \text{-} \mathbb X) \subseteq \mathrm{Exc}(\mathrm{coh} \text{-} \mathbb X)$. By Proposition~\ref{domestic_order}, $\mathcal A < (\prescript{\perp}{}{\mathcal B} \cap \mathrm{tor}_{\lambda_i} \text{-} \mathbb X)^{\perp}$ if and only if $\mathrm{cox}(\mathcal A) < c_{\mathbb X} \cdot \mathrm{cox}(\prescript{\perp}{}{\mathcal B} \cap \mathrm{tor}_{\lambda_i} \text{-} \mathbb X)^{-1}$, i.e. $\mathrm{cox}(\mathcal A) \cdot \mathrm{cox}(\prescript{\perp}{}{\mathcal B} \cap \mathrm{tor}_{\lambda_i} \text{-} \mathbb X)<c_{\mathbb X}$.

   (3) By Proposition~\ref{thick_uniserial}, both $\prescript{\perp}{}{\mathcal A} \cap \mathrm{tor}_{\lambda_i} \text{-} \mathbb X$ and $\prescript{\perp}{}{\mathcal B} \cap \mathrm{tor}_{\lambda_i} \text{-} \mathbb X$ can be generated by an exceptional sequence. So $\mathcal A \leq \mathcal B$ if and only if $\prescript{\perp}{}{\mathcal B} \cap \mathrm{tor}_{\lambda_i} \text{-} \mathbb X \leq \prescript{\perp}{}{\mathcal A} \cap \mathrm{tor}_{\lambda_i} \text{-} \mathbb X$, if and only if $\mathrm{cox}(\prescript{\perp}{}{\mathcal B} \cap \mathrm{tor}_{\lambda_i} \text{-} \mathbb X) \leq \mathrm{cox}(\prescript{\perp}{}{\mathcal A} \cap \mathrm{tor}_{\lambda_i} \text{-} \mathbb X)$ by Proposition~\ref{domestic_order}.

   (4) It follows directly from Proposition~\ref{uniserial}(2) and (3).
\end{proof}

\begin{rmk}
   {\rm The poset $\mathrm{Wid}(\mathrm{coh} \text{-} \mathbb X)^{\vec c}$ is a lattice. If $\mathcal A$, $\mathcal B \in \mathrm{Wid}(\mathrm{coh} \text{-} \mathbb X)^{\vec c}$, obviously $\mathcal A \cap \mathcal B \in \mathrm{Wid}(\mathrm{coh} \text{-} \mathbb X)^{\vec c}$. By Proposition \ref{contain_non_exc}, we have that $\langle \mathcal A,\mathcal B \rangle \in \mathrm{Wid}(\mathrm{coh} \text{-} \mathbb X)^{\vec c}$. }
\end{rmk}

\section{Two examples}

In this section we give two examples. In the first example, we recover the result in \cite[Subsection~4.1]{krause2017derived} via Theorem~\ref{main}. The second shows the poset structure of $\mathrm{Wid}(\mathrm{coh}\text{-}\mathbb X(2))$, which is a weighted projective line of domestic type.

\begin{exa} [{\cite[Subsection~4.1]{krause2017derived}}]
We view the projective line $\mathbb P^1(\mathbb C)$ as a weighted projective line graded by $\mathbb Z$. Note that $\vec c=1$. By Theorem~\ref{c_inv}, we have $\mathrm{Wid}(\mathrm{coh} \text{-}\mathbb P^1)^{\vec c}=\mathrm{Wid}(\mathrm{coh}_0 \text{-} \mathbb P^1) \cup \{\mathrm{coh} \text{-}\mathbb P^1 \}$. It follows that the lattice $\mathrm{Wid}(\mathrm{coh} \text{-}\mathbb P^1)^{\vec c}$ is isomorphic to the lattice of specialized closed subsets of $\mathbb P^1$.

Since the exceptional objects of $\mathrm{coh} \text{-} \mathbb P^1$ are exactly $\{\mathcal O(l) \ | \ l \in \mathbb Z\}$, we have that $\mathrm{Exc}(\mathrm{coh} \text{-}\mathbb P^1)=\{\langle \mathcal O(l) \rangle | l \in \mathbb Z \} \cup \{0 \} \cup \{\mathrm{coh} \text{-}\mathbb P^1 \}$. It follows that $\mathrm{Exc}(\mathrm{coh} \text{-}\mathbb P^1) \cong \mathbb Z$, where $\mathbb Z$ is the lattice given by the Hasse diagram (\ref{lattice_z}).

By Theorem~\ref{main}, we have an isomorphism of lattices 
\[
   \mathrm{Wid}(\mathrm{coh} \text{-}\mathbb P^1) \cong \{\text{specialization closed subsets of } \mathbb P^1 \} \coprod \mathbb Z.
\]
\end{exa}

\begin{exa}
Let $\mathbb X$ be a weighted projective line of type $(2)$. Then $L({\bf p})=\langle \vec{x}_1, \vec{x}_2 | 2\vec{x}_1=\vec{x}_2 \rangle $. The homogeneous coordinate ring $S=k[x_1,x_2]$, with $\mathrm{deg}(x_1)=\vec{x}_1$, $\mathrm{deg}(x_2)=\vec{x}_2$. We have $D^b(\mathrm{coh} \text{-} \mathbb X) \cong D^b(kQ)$, where $Q$ is the quiver
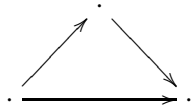
\[\xymatrix{
 & \cdot \ar[rd] & \\
\cdot \ar[ur] \ar[rr] & &\cdot
}\]
We calculate the poset $\mathrm{Wid}(\mathrm{coh} \text{-} \mathbb X)$. 

By Theorem \ref{domestic}, line bundles $\{\mathcal O, \mathcal O(\vec{x}_1), \mathcal O(\vec{x}_2)\}$ form a $\tau$-slice of the AR-quiver. So all the indecomposable bundles are line bundles. We notice that $(\mathcal O, \mathcal O(\vec l))$ forms an exceptional sequence if and only if $\vec c+\vec w \leq \vec l \leq \vec c$, i.e. $-\vec{x}_1 \leq \vec l \leq \vec{x}_2$. So the set $\mathrm{Exc}(\mathrm{coh} \text{-} \mathbb X)$ is clear: let $\mathcal T_0=\langle \mathcal O \rangle$, $\mathcal T_1=\langle \mathcal O, \mathcal O({\vec x_1)} \rangle$, $\mathcal T_2=\langle \mathcal O, \mathcal O({\vec x_2)} \rangle$, then
\[
\mathrm{Exc}(\mathrm{coh} \text{-} \mathbb X)=\{\mathcal T_i (n \vec x_1) | i=0,1,2; n \in \mathbb Z \} \cup \{\langle S_{\infty,0} \rangle \} \cup \{\langle S_{\infty,1} \rangle \} \cup \{0 \} \cup \{\mathrm{coh} \text{-} \mathbb X\},
\]
where $S_{\infty,i}$ is defined to be the cokernel of $\mathcal O((i-1)\vec x_1) \to \mathcal O(i \vec x_1)$. And the poset structure is also clear: $\langle S_{\infty,1} \rangle <\mathcal T_1(2n \vec x_1)$, $\langle S_{\infty,0} \rangle <\mathcal T_1((2n-1) \vec x_1)$, $\mathcal T_0(n \vec x_1)< \mathcal T_1(n \vec x_1)$,$\mathcal T_0(n \vec x_1)< \mathcal T_1((n-1) \vec x_1)$, $\mathcal T_2 >\mathcal T_0(2n \vec x_1)$, $\mathcal T_2 (\vec x_1)>\mathcal T_0((2n-1) \vec x_1)$, $n \in \mathbb Z$. The Hasse diagram of $\mathrm{Exc}(\mathrm{coh} \text{-} \mathbb X)\setminus (\{0\} \cup \{ \mathrm{coh} \text{-} \mathbb X\} ) $ is as follows:

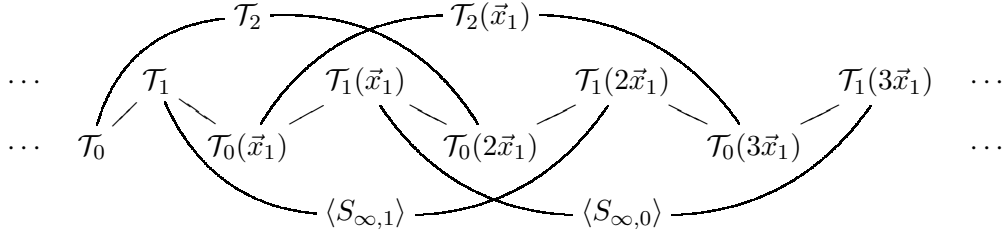
\xymatrix@=1.5ex{
 & & & \mathcal T_2 \ar @/_1.3pc/ @{-} [lldd] \ar @/^1.5pc/ @{-} [rrdd] & & \mathcal T_2(\vec x_1) \ar @/_1.5pc/ @{-} [lldd] \ar @/^1.5pc/ @{-} [rrdd] & & & & & \\
\cdots & & \mathcal T_1 \ar@{-} [dl] \ar@{-} [dr] & & \mathcal T_1(\vec x_1) \ar@{-} [dl] \ar@{-}[dr] & & \mathcal T_1(2\vec x_1) \ar@{-} [dl] \ar@{-} [dr] & & \mathcal T_1(3\vec x_1) \ar@{-} [dl]  & \cdots \\
\cdots & \mathcal T_0 & & \mathcal T_0(\vec x_1) & & \mathcal T_0(2\vec x_1) & & \mathcal T_0(3\vec x_1) & & \cdots \\
 & & & & \langle S_{\infty,1} \rangle \ar @/^1.5pc/ @{-} [lluu] \ar @/_1.5pc/ @{-} [rruu] & & \langle S_{\infty,0} \rangle \ar @/^1.5pc/ @{-} [lluu] \ar @/_1.5pc/ @{-} [rruu] & & & &
}

By Theorem~\ref{c_inv}, we have
\[
\mathrm{Wid}(\mathrm{coh} \text{-} \mathbb X)^{\vec c}=\prod_{\lambda \in \mathbb X}\mathrm{Wid}(\mathcal U_{\lambda}) \cup \{ S_{\infty,0}^{\perp} \} \cup \{ S_{\infty,1}^{\perp} \} \cup \{\mathrm{coh} \text{-} \mathbb X \}
\]
as a set. We know that $S_{\infty,0}^{\perp}=\mathcal T_2$, $S_{\infty,1}^{\perp}= \mathcal T_2 (\vec x_1)$.

Since $\mathcal U_\infty$ is a tube of rank 2, we have
\[
\mathrm{Wid}(\mathcal U_{\infty})=\{\{0\}, \langle S_{\infty,0} \rangle, \langle S_{\infty,1} \rangle, \langle S_{\infty,0}^{[2]} \rangle, \langle S_{\infty,1}^{[2]} \rangle, \mathcal U_{\infty} \}.
\]
We know that $\langle S_{\infty,0}^{[2]} \rangle= \mathcal T_2 \cap \mathcal U_{\infty}$, $\langle S_{\infty,1}^{[2]} \rangle = \mathcal T_2 (\vec x_1) \cap \mathcal U_{\infty}$. Then the partial orders on $\mathrm{Wid}(\mathrm{coh} \text{-} \mathbb X)^{\vec c}$ follows easily from Proposition~\ref{c_inv_order}, and the poset structure of $\mathrm{Wid}(\mathcal U_{\infty})$ is trivial.
\end{exa}

\section*{Acknowledgements}

The author thanks Professor Henning Krause for suggesting the author to work on this problem, and for sharing his notes on perpendicular calculus for hereditary categories. And the author acknowledges Professor Xiao-Wu Chen for many helpful discussions and advices. Finally, the author thanks the China Scholarship Council for the financial support.

%A construction of $\langle \mathcal X \rangle$ is as follows: Let
%\[
%\mathcal S_1=\Set{X \in \mathcal H}{\begin{gathered} \exists \mathrm{\ an\ exact\ sequence\ formed\ by\ } X \mathrm{\ and\ }A, B \in \mathcal {X}, \\ \mathrm{or\ } \exists\ Y \in \mathcal X, \mathrm{\ s.t.\ } X \mathrm{\ is\ a\ direct\ summand\ of\ } Y.
%\end{gathered}},
%\]
%\[
%\mathcal S_2=\Set{X \in \mathcal H}{\begin{gathered} \exists \mathrm{\ an\ exact\ sequence\ formed\ by\ } X \mathrm{\ and\ }A, B \in \mathcal{S}_1, \\ \mathrm{or\ } \exists\ Y \in \mathcal{S}_1, \mathrm{\ s.t.\ } X \mathrm{\ is\ a\ direct\ summand\ of\ } Y.
%\end{gathered}},
%\]
%\[
%...
%\]
%then
%\[
%\langle \mathcal X \rangle = \bigcup_{i=1}^\infty \mathcal S_i
%\]. So in a length category, the composition factors of $\langle \mathcal X \rangle$ are exactly the composition factors of $\mathcal X$.

\end{document}